\documentclass[a4paper,11pt]{amsproc}

\usepackage{enumerate}
\usepackage{csquotes}

\usepackage{xcolor, soul}
\sethlcolor{yellow}

\usepackage{xargs}
\usepackage{amsmath,amssymb,amsfonts,amsthm,mathdots}
\usepackage{mathtools, todonotes}

\usepackage[colorlinks]{hyperref}
\usepackage{amsrefs}

\usepackage[onehalfspacing]{setspace}

\newcommandx{\set}[2][2=\empty]{\{#1\ifx#2\empty\else\,|\,#2\fi\}}
\newcommand{\norm}[1]{\lvert#1\rvert}
\newcommand{\card}[1]{\left\lvert#1\right\rvert}
\newcommandx{\gensubgrp}[2][2=\empty]{\langle#1\ifx#2\empty\else\,|\,#2\fi\rangle}
\newcommandx{\gensubsp}[2][2=\empty]{\langle#1\ifx#2\empty\else\,|\,#2\fi\rangle}
\newcommand{\ints}{\mathbb{Z}}
\newcommand{\finfield}{\mathbb{F}}
\DeclareMathOperator{\diag}{diag}
\newcommand{\floor}[1]{\lfloor#1\rfloor}
\DeclareMathOperator{\fix}{fix}
\newcommand{\trivgrp}{\mathbf{1}}
\newcommand{\M}{\mathbf{M}}
\DeclareMathOperator{\End}{End}
\DeclareMathOperator{\Sp}{Sp}
\DeclareMathOperator{\PSp}{PSp}
\DeclareMathOperator{\GO}{GO}
\DeclareMathOperator{\GU}{GU}
\DeclareMathOperator{\PSU}{PSU}
\DeclareMathOperator{\SO}{SO}
\DeclareMathOperator{\PSO}{PSO}
\DeclareMathOperator{\POmega}{P\Omega}
\DeclareMathOperator{\Aut}{Aut}
\newcommand{\freegrp}{\mathbf F}
\DeclareMathOperator{\GL}{GL}
\DeclareMathOperator{\SL}{SL}
\DeclareMathOperator{\PSL}{PSL}
\DeclareMathOperator{\PGL}{PGL}
\DeclareMathOperator{\im}{im}
\newcommand{\proj}{\mathbf P}
\DeclareMathOperator{\soc}{soc}
\DeclareMathOperator{\id}{id}
\DeclareMathOperator{\tr}{tr}
\DeclareMathOperator{\SU}{SU}
\DeclareMathOperator{\rad}{rad}
\DeclareMathOperator{\N}{N}
\newcommand{\rest}[1]{\left.#1\right\rvert}
\newcommand{\C}{\mathbf{C}}
\theoremstyle{plain}
\newtheorem{theorem}{Theorem}
\newtheorem{lemma}{Lemma}[section]
\newtheorem{corollary}[lemma]{Corollary}
\newtheorem{conj}[lemma]{Conjecture}

\theoremstyle{definition}
\newtheorem{remark}[lemma]{Remark}

\title[Mixed identities for finite groups]{The length of mixed identities for finite groups}
\author{Henry Bradford}
\address{H.~Bradford, Univ.\ of Cambridge, Cambridge
	CB3 0WB, Unites Kingdom}
\email{hb470@cam.ac.uk}
\author{Jakob Schneider}
\address{J.~Schneider, TU Dresden, 01062 Dresden, Germany}
\email{jakob.schneider@tu-dresden.de}
\author{Andreas Thom}
\address{A.~Thom, TU Dresden, 01062 Dresden, Germany}
\email{andreas.thom@tu-dresden.de}

\begin{document}
	\begin{abstract}
		We prove that there exists a constant $c>0$ such that any finite group having no non-trivial mixed identity of length $\leq c$ is an almost simple group with a simple group of Lie type as its socle. Starting the study of mixed identities for almost simple groups, we obtain results for groups with socle $\PSL_n(q)$, $\PSp_{2m}(q)$, $\POmega_{2m-1}^\circ(q)$, and $\PSU_n(q)$ for a prime power $q$. For such groups, we will prove rank-independent bounds for the length of a shortest non-trivial mixed identity, depending only on the field size $q$.
	\end{abstract}
	
	\maketitle
	
	\tableofcontents
	
	\section{Introduction}
	
	In this article we study \emph{identities with constants}  
	(also called \emph{mixed identities}) for finite groups. A word with constants in a finite group $G$ is an element of the free product $w\in G\ast\freegrp_r$. Note  that $w$ induces a map $w\colon G^r\to G$ by evaluation. A non-trivial word with constants $w$ is called an \emph{identity with constants} or a \emph{mixed identity} for $G$ if and only if $w(g_1,\ldots,g_r)=1_G$ for all choices of the $g_i\in G$ ($i=1,\ldots,r$). Without loss of generality, we will restrict our attention almost only to the case $r=1$, see Lemma~\ref{lem:red_one_var_cs}.
	
	The study of word maps with and without constants on finite and algebraic groups has seen a lot of progress in the past decades, see for example \cites{elkasapythom2014goto, gordeevkunyavskiiplotkin2016word, gordeevkunyavskiiplotkin2018word, guralnickliebeckobrienshalevtiep2018surjective, klyachkothom2017new, larsen2004word, larsenshalev2009word, larsenshalevtiep2012waring, lubotzky2014images, nikolovschneiderthom2018some, schneiderthom2022word, schneiderthom2021word, schneider2019phd, thom2017length,bradford2019short, bradford2019lie, gordeev1997freedom} and the references therein.
	
	The length of a shortest mixed identity that is satisfied by the group $G$ is a natural measure of complexity for $G$ and our motivation is to understand which groups do not satisfy a short mixed identity. The first part of our main result says that any group that does not satisfy a mixed identity of length at most $8$ must be almost simple. The second part concerns the study of mixed identities for almost simple groups and describes lower and upper bounds for the length of shortest
	mixed identities for particular families. Our study focuses on the families of groups of Lie type with socle $\PSL_n(q)$, $\PSp_{2m}(q)$, $\POmega_{2m-1}^\circ(q)$, or $\PSU_n(q)$, 
	that is, excluding orthogonal groups in even dimension.
	
	It is subject of ongoing work to cover also the family $\POmega_{2m}^\pm(q)$ for $m \geq 2$ and the exceptional groups of Lie type.
	
	As a consequence of \cite{tomanov1985generalized}, we can identify two families: $\PSp_{2m}(q)$, $m \geq 2$ $\POmega_{2m-1}^\circ(q)$, $m \geq 3$ odd, or $m \geq 3$ arbitrary and $q\equiv 1$ mod $4$, for which almost simple groups with corresponding socle satisfy a mixed identity of bounded length. The same phenomenon occurs for alternating groups. 
	By the work of Jones \cite{jones1974varieties}, there does not exist an infinite family of pairwise nonisomorphic almost-simple groups satisfying identities (without constants) of bounded length. As the cases of the alternating, symplectic and orthogonal groups illustrate, this is no longer true for identities with constants, and one would like to have a classification of families of nonabelian finite simple groups satisfying identities with constants of bounded length.
	
	Now, even if we know that a family of simple groups, say $\PSL_2(q)$, does satisfy a lower bound $\Omega(q)$ for the length of a shortest mixed identity, it so happens that almost simple groups with socle $\PSL_2(q)$ can satisfy much shorter mixed identities. In fact, we show that $\Aut(\PSL_2(q))$ satisfies a mixed identity of length $O(ep)$ with $q=p^e$. In this connection, 
	we note that our methods also yield mixed 
	identities of bounded length for the groups 
	$\PSO_{2m-1}^\circ(q)$ for $m \geq 4$ even and 
	$q \equiv 3$ mod $4$, 
	which involve constants 
	lying outside the simple groups $\POmega_{2m-1}^\circ(q)$. It is as yet unclear how the length of the shortest mixed identity for the latter groups should behave, see Section \ref{comments}. These examples indicate that the study of mixed identities for almost simple groups is a subtle topic with some unexpected phenomena. 
	
	Let us now describe the results of this paper in more detail.
	Our first main result is the following theorem.
	
	\begin{theorem}\label{thm:red_almst_smpl_grp}
		For $G$ a finite group, there exists a mixed identity of length at most $8$, or $G$ is almost simple. 
		
		In the latter case, there is an absolute constant $c>0$, so that, if $G$ has no mixed identity of length $\leq c$, then the socle of $G$ is a simple group of Lie type, different from $\PSp_{2m}(q)$, for $m \geq 2$, and $\POmega^\circ_{2m-1}(q)$, for $m\geq 3$ odd, or $m\geq 3$ arbitrary and $q\equiv 1$ mod $4$. 
	\end{theorem}
	
	The characterization of almost simple groups that admit mixed identities of bounded length proceeds family by family, where we only have partial results so far. First of all, note the following, which is a consequence of Lemma~\ref{lem:nor_subgrp} below.
	
	\begin{lemma}\label{lem:nor_subgrp2}
		Let $G\leq H$ be an inclusion of almost simple groups with  socle $S$. If $G$ has a  mixed identity of length $l$, then $H$ has a mixed identity of length at most $2l$.
	\end{lemma}
	
	This applies to groups with socle $A_n$ by direct inspection (for a $3$-cycle $\sigma$, $w(x)= [x,\sigma]^{30}\in A_n\ast\gensubgrp{x}$ is a mixed identity for $A_n$) \label{fct:bd_ids_alt_grps} and to groups with socle $\PSp_{2m}(q)$, for $m \geq 2$, or $\POmega^{\circ}_{2m-1}(q)$, for $m \geq 3$ odd or $q \equiv 1$ mod $4$, as a consequence of results of Tomanov \cite{tomanov1985generalized}. For convenience, we reproduce his results with short and self-contained proofs.
	
	The first interesting case is the case of almost simple groups with socle $\PSL_2(q).$ In this case, we get a complete answer as follows. 
	Let $F$ denote the Frobenius automorphism 
	$x \mapsto x^p$ of the finite field of order 
	$q=p^e$, and also the induced automorphism 
	of $\PGL_n(q)$.
	
	\begin{theorem} \label{thm:psl2}
		Let $G$ be an almost simple group with socle $\PSL_2(q)$ and $q=p^e$ for a prime number $p$. Let $f\mid e$ be the smallest natural number, such that $F^f \in G.$ Then the length of a shortest mixed identity of $G$ is $\Theta(\frac{e}{f} p^f)$.
	\end{theorem}
	
	In the case of almost simple groups with socle $\PSL_n(q)$ for $n \geq 3$, we only have partial results. Note however that the implied constants in the next theorem are independent of the rank.
	
	\begin{theorem} \label{thm:psl}
		Let $G$ be an almost simple group with socle $\PSL_n(q)$ and $q=p^e$ for a prime number $p$. Then, $G$ has a mixed identity of length $O(q)$. Moreover, if $G \le \PGL_n(q)\rtimes\Aut(\finfield_q)$, $F$ is the Frobenius automorphism as above, and $f\mid e$ is the smallest natural number such that $F^f \in G$, then any mixed identity of $G$ is of length $\Omega(\frac{e}{f} p^f).$
	\end{theorem}
	
	
	Note that this is contrast to the minimal length of identities without constants for $\PSL_n(q)$ which are known to be bounded from below by $q^{\floor{n/2}}$ and bounded from above by $O(q^{\floor{n/2}}\log(q)^{O_n(1)})$, by results of the first and the third author \cite{bradford2019lie}. In case $n\geq 3$, we do not know yet what effect the transpose-inverse has on the length of shortest mixed identities.
	
	Our result for the family $\PSU_n(q)$ is less refined and reads as follows:
	
	\begin{theorem}\label{thm:unitry_groups_main_thm}
		Let $G$ be an almost simple group with socle $\PSU_n(q)$. Then, $G$ has a mixed identity of length $O(q^2)$. Moreover, any mixed identity for $\PSU_n(q)$, even with constants from $\PGL_n(q^2)$, is of length $\Omega(q)$.
	\end{theorem}
	
	Even though there exist mixed identities of bounded length for $\PSp_{2m}(q)$, our methods allow for some more refined understanding of the structure of the mixed identities that can occur. A constant appearing in a word with constants is called \emph{critical} if its removal leads to cancellation of the variables.
	
	\begin{theorem} \label{thm:sp}
		Let $q$ be a prime power and $m\geq2$. A shortest mixed identity for $\PSp_{2m}(q)$ without critical constants which lift to involutions in $\Sp_{2m}(q)$ is of length $\Theta(q)$ for $q$ odd. For $q$ even it lies in $\Omega(q)$.
	\end{theorem}
	
	These results resemble analogous results of Tomanov \cite{tomanov1985generalized} and Gordeev \cite{gordeev1997freedom}, for algebraic groups over infinite fields. 
	One may also use these results for algebraic groups, combined with the 
	Schwartz-Zippel Lemma, 
	to prove lower bounds on the lengths of mixed 
	identities for finite groups of Lie type. 
	Indeed, we shall exploit these methods in a 
	forthcoming paper. 
	However, the results obtained by such methods would not be uniform in the rank, as our bounds here are. 
	
	The article is organized as follows. After the introduction we have a section covering basic observations. After that we have one section for each family of simple groups that is covered, i.e. $\PSL_n(q)$, $\PSp_{2m}(q)$, $\POmega_{2m-1}^\circ(q)$, and $\PSU_n(q)$. Various arguments for $\PSp_{2m}(q)$ and $\PSU_n(q)$ will follow the same lines as the prototypical argument for $\PSL_n(q)$ and we recommend the reader to read this case first. We end the paper with a section on further remarks and goals for the future.
	
	We apply the main results of this paper in \cite{bradford2023length} and answer a question from \cite{bradford2021lawless} on the length of non-solutions to equations with constants in linear groups.
	
	\section{Basic observations}
	
	Let $G$ be a finite group and $C\geq G$ be the overgroup of possible constants. Recall that a mixed identity $w\in C\ast\freegrp_r$ is called a \emph{shortest} mixed identity for $G$ with constants from $C$ if there is no shorter one, i.e.\ for $v\in C\ast\freegrp_r$ another mixed identity, we have $\norm{w}\leq\norm{v}$. Here $\norm{w}=l$ measures the length of the fixed word
	$$
	w=c_0 x_{i(1)}^{\varepsilon(1)} c_1 \cdots c_{l-1} x_{i(l)}^{\varepsilon(l)} c_l\in C\ast\freegrp_r,
	$$ 
	where $\varepsilon(j)=\pm1$ ($j=1,\ldots,l$) and $c_j\in C$ ($j=0,\ldots,l$). We always assume that
	$$
	x_{i(j)}^{\varepsilon(j)}=x_{i(j+1)}^{-\varepsilon(j+1)}
	$$ 
	for $j=1,\ldots,l-1$ implies $c_j\neq 1_C$; i.e.\ $w$ is \emph{reduced}. The word $w$ is called \emph{cyclically reduced} if $x_{i(l)}^{\varepsilon(l)}=x_{i(1)}^{-\varepsilon(1)}$ implies $c_lc_0\neq 1_C$.
	
	The first basic observation is that a shortest mixed identity for $G$ with constants from some given group $C$ is always cyclically reduced:
	
	\begin{lemma}\label{lem:shrt_id_cyc_red}
		Let $w\in C\ast\freegrp_r$ be a shortest mixed identity for $G$. Then $w$ is cyclically reduced.
	\end{lemma}
	
	\begin{proof}
		We can write $w$ as $w=u^{-1}vu$, where $u,v\in C\ast\freegrp_r$ and $v$ is cyclically reduced. If $w$ is a mixed identity for $G$, then $w(g_1,\ldots,g_r)=1_C=v(g_1,\ldots,g_r)^{u(g_1,\ldots,g_r)}$ for all $g_1,\ldots,g_r\in G$. Thus $v$ is also a mixed identity for $G$ whose length is at most $\norm{w}$. But we cannot have $v=c$ for a $c\in C$, since then if $c\neq 1_C$, we have $w(1_G,\ldots,1_G)=v(1_G,\ldots,1_G)^{u(1_G,\ldots,1_G)}=c^{u(1_G,\ldots,1_G)}\neq 1_C$. If $c=1_C$, then $w$ would be trivial. Hence $v\in C\ast\freegrp_r\setminus C$ is a shortest mixed identity and $u\in C$. The proof is complete.
	\end{proof}
	
	Fix a reduced word $w=c_0 x_{i(1)}^{\varepsilon(1)} c_1 \cdots c_{l-1} x_{i(l)}^{\varepsilon(l)} c_l\in C\ast\freegrp_r$. Define the sets of indices $J_0(w),J_+(w),J_-(w)\subseteq\set{1,\ldots,l-1}$ by $J_0(w)\coloneqq\set{j}[i(j)\neq i(j+1)]$, $J_+(w)\coloneqq\set{j}[i(j)=i(j+1)\text{ and }\varepsilon(j)=\varepsilon(j+1)]$, and $J_-(w)\coloneqq\set{j}[i(j)=i(j+1)\text{ and }\varepsilon(j)=-\varepsilon(j+1)]$, which partition the set $\set{1,\ldots,l-1}$. The constants $c_1,\ldots,c_{l-1}\in C$ are called \emph{intermediate constants}. The constants $c_j$ with $j\in J_-(w)$ are called \emph{critical constants}\label{def:crit_consts}.
	
	We have the following second observation which guarantees that we need to consider only words with one variable $x$:
	
	\begin{lemma}\label{lem:red_one_var_cs}
		Let $w=c_0 x_{i(1)}^{\varepsilon(1)} c_1 \cdots c_{l-1} x_{i(l)}^{\varepsilon(l)} c_l\in C\ast\freegrp_r$ be reduced. Then, assuming $l=\norm{w}\leq\card{G}$, there is a substitution $s\colon x_i\mapsto g_{-i} x g_i$ for $g_{\pm i}\in G$ ($i=1,\ldots,r$) such that in 
		$$
		w'\coloneqq w(s(x_1),\ldots,s(x_r))=c_0' x^{\varepsilon(1)} c_1' \cdots c_{l-1}' x^{\varepsilon(l)} c_l'\in C\ast\gensubgrp{x}
		$$ 
		we have $c_j'\neq 1_C$ for $j=1,\ldots,l-1$.
	\end{lemma}
	
	\begin{proof}
		We have that $c_j'=g_{\varepsilon(j) i(j)}^{\varepsilon(j)}c_j g_{-\varepsilon(j+1) i(j+1)}^{\varepsilon(j+1)}$ ($j=1,\ldots,l-1$). So among all the possible $\card{G}^{2r}$ choices for the constants $g_{\pm i}$ ($i=1,\ldots,r$), each condition $c_j'\neq 1_C$ for $j\in J_0(w)\cup J_+(w)$ rules out at most $\card{G}^{2r-1}$ tuples. If $j\in J_-(w)$, then we must have 
		$$
		c_j'=c_j^{g_{\varepsilon(j)i(j)}^{-\varepsilon(j)}}\neq 1_C,
		$$ since $c_j\neq 1_C$ by assumption. Hence, if $l-1<\card{G}$, i.e.\ $\card{G}^{2r-1}(l-1)<\card{G}^{2r}$, by counting, there must be one tuple $(g_i)_{i=\pm 1}^{\pm r}$ such that $c_j'\neq 1_C$ ($j=1,\ldots,l-1$).
	\end{proof}
	
	\begin{remark}\label{rem:cyc_red_cs}
		If $w$ is cyclically reduced and $l<\card{G}$, we can also guarantee $w'$ to be cyclically reduced. We have
		$$
		c_l'c_0'=g_{\varepsilon(l)i(l)}^{\varepsilon(l)}c_lc_0 g_{-\varepsilon(1)i(1)}^{\varepsilon(1)}.
		$$ 
		If $i(l)=i(1)$ and $\varepsilon(1)=-\varepsilon(l)$, then as $w$ is cyclically reduced, we must have $c_lc_0\neq 1_C$ and hence 
		$$
		c_l'c_0'=(c_lc_0)^{g_{\varepsilon(l)i(l)}^{-\varepsilon(l)}}\neq 1_C.
		$$
		In the opposite case, we rule out at most $\card{G}^{2r-1}$ further tuples. But $\card{G}^{2r-1}l<\card{G}^{2r}$, so there is a legal choice for $(g_i)_{i=\pm 1}^{\pm r}$.
	\end{remark}
	
	From this we get the following immediate non-optimal corollary with a short proof:
	
	\begin{corollary}\label{cor:shrt_id_one_var}
		There is a shortest mixed identity $w\in C\ast\gensubgrp{x}$ for $G$ with only one variable and all intermediate constants non-trivial. It is cyclically reduced and of length $\leq\card{G}$.
	\end{corollary}
	
	\begin{proof}
		Since $x^{\card{G}}$ is a mixed identity for $G$, a shortest mixed identity $w\in C\ast\freegrp_r$ for $G$ has length at most $\card{G}$. Clearly, we may assume $G\neq\trivgrp$, since otherwise $w=x$ is a shortest mixed identity. Hence either $w=(xc)^{\card{G}}$ (for $c\in G\setminus\trivgrp\subseteq C$) is a shortest mixed identity with all intermediate constants non-trivial, or there is a shortest mixed identity $w\in C\ast\freegrp_r$ of length $<\card{G}$, which by Lemma~\ref{lem:shrt_id_cyc_red} is cyclically reduced. Applying Lemma~\ref{lem:red_one_var_cs} and Remark~\ref{rem:cyc_red_cs} gives a shortest mixed identity of length $<\card{G}$ with only one variable and all intermediate constants non-trivial; it is cyclically reduced.
	\end{proof}
	
	The next lemma proves that there are no short identities of length less than four if the groups $G$ and $C$ fulfill some mild assumptions.
	
	\begin{lemma}\label{lem:no_shrt_ids}
		Let $w\in C\ast\gensubgrp{x}$ be of length $\norm{w}=l$. Let $G\neq\trivgrp$ be non-abelian and $\C_C(G)=\trivgrp$ if $l\leq 2$, and let $C=G$, $\card{G}$ be even when $l=3$. Then $w$ is not a mixed identity for $G$ with constants from $C$. 
	\end{lemma}
	
	\begin{proof}
		Clearly, any word of length one induces an injective map, so cannot be constant if $G\neq\trivgrp$. If $l=2$, then, up to rotation and replacing $x$ by $x^{-1}$, either (a) $w=xcxc^{-1}$ or (b) $w=xcx^{-1}c^{-1}$ (for $c\in C$). In Case~(a), if $w$ induces the trivial map, we must have $g^{c^{-1}}=g^{-1}$ for all $g\in G$, in particular, $g\mapsto g^{-1}$ would an automorphism of $G$, so $G$ would be abelian, which is not the case. In Case~(b), if $w$ is trivial on $G$, then $c\in\C_C(G)=\trivgrp$, so that $c=1_G$ and $w$ would not be reduced. If $l=3$, up to rotation and replacing $x$ by $x^{-1}$, we have that (a) $w=xaxbx^{-1}c$; or (b) $w=xaxbxc$  with $abc=1_G$.
		Hence, if $w$ is a mixed identity, then $xaxbx^\varepsilon cc^{-1}x^{-\varepsilon}=xaxb=c^{-1}x^{-\varepsilon}$. Thus we get $xaxa=c^{-1}x^{-\varepsilon}b^{-1}a$. This cannot hold when $\card{G}$ is even, since then there is an element $g\in G$ of order two, so that $ga^{-1}.a.ga^{-1}.a=1_G=a^{-1}.a.a^{-1}.a$, but $a^{-1}\neq ga^{-1}$ and both are from $G$, since by assumption $C=G$. However, the map $x\mapsto c^{-1}x^{-\varepsilon}b^{-1}a$ is injective, which gives a contradiction.
	\end{proof}
	
	\section{Reduction to almost simple groups of Lie type}
	
	In this section, we prove  Theorem~\ref{thm:red_almst_smpl_grp}, modulo the statements about symplectic and orthogonal groups. The proof is based on the following two lemmas.
	
	\begin{lemma}\label{lem:ab_prod_cse_shrt_id}
		If $G$ has a non-trivial center, then it satisfies the mixed identity $[x,c]\in G\ast\gensubgrp{x}$ for $c\in \C(G)\setminus\trivgrp$ of length $2$. Similarly, if $G$ is a non-trivial direct product $G=A\times B$, then it satisfies the mixed identity $[a^x,b]\in G\ast\gensubgrp{x}$, for non-trivial $a\in A \times\trivgrp$, $b \in\trivgrp\times B$, which is of length $4$.
	\end{lemma}
	
	\begin{proof}
		A trivial computation.
	\end{proof}
	
	\begin{lemma}\label{lem:nor_subgrp}
		Let $G$ be a finite group and let $\trivgrp\neq N\trianglelefteq G$ be a normal subgroup. Suppose $N$ has a mixed identity with constants in $G$ of length $l$. Then $G$ has a mixed identity of length at most $2l$. 
	\end{lemma}
	
	\begin{proof}
		Let $w\in G\ast\freegrp_r$ be a mixed identity for $N$ of length $l$. By Corollary~\ref{cor:shrt_id_one_var}, we may assume that $w$ is of length $l\leq\card{N}$, has only one variable, and all intermediate constants of $w$ are non-trivial, i.e.\ 
		$$
		w=c_0 x^{\varepsilon(1)} c_1\cdots c_{l-1} x^{\varepsilon(l)} c_l
		$$ 
		with $c_j\neq1_G$ for $1\leq j\leq l-1$. For $n\in N\setminus\trivgrp$ set $v\coloneqq w(n^x)$. Clearly, $\im(v)\subseteq\im(w)=\trivgrp$, so it suffices to check that $v\neq 1_G$ is non-trivial in $G\ast\gensubgrp{x}$. But there is no cancellation in $v$, since $x^{\varepsilon(j)} c_j x^{\varepsilon(j+1)}$ becomes $x^{-1} n^{\varepsilon(j)}x c_j x^{-1} n^{\varepsilon(j+1)}x$, so $v$ is non-trivial of length $2l$.
	\end{proof}
	
	\begin{remark}
		It is clear that if $w\in\freegrp_r$ is an identity for the group $G$, then $w$ is also an identity for every subgroup $H$ and every quotient $Q$ of $G$. In particular, the length of the shortest identities (without constants) for $H$ and $Q$ are at most the length of a shortest identity for $G$. The analogous statements for mixed identities are false: Let $H=\PSL_2(q)$ and $G=\PSL_2(q)\times C_2$. Then, $H$ is a subgroup and a quotient of $G$; by Lemma~\ref{lem:ab_prod_cse_shrt_id}, $G$ has a mixed identity of length $2$, whereas by Theorem~\ref{thm:psl} the length of a shortest mixed identity for $H$ is $\Theta(q)$. 
	\end{remark}
	
	Let's recall Fitting's structure theorem.
	
	\begin{theorem}[Fitting]\label{thm:fitting}
		Let $G$ be a non-trivial finite group and suppose $G$ has no non-trivial abelian normal subgroup. Then there exist positive integers $k$ and $l_1,\ldots,l_k$ and distinct non-abelian finite simple groups $H_1,\ldots,H_k$ such that: 
		$$
		S=\prod_{i=1}^k H_i ^{l_i}\trianglelefteq G\leq\prod_{i=1}^k \Aut(H_i)^{l_i}\wr S_{l_i}=\Aut(S)\text{.}
		$$
		Here the socle $S=\soc(G)$ of $G$ is the product $\prod_{i=1}^k H_i^{l_i}$.
	\end{theorem}
	
	
	Now we are ready to prove Theorem~\ref{thm:red_almst_smpl_grp}:
	
	\begin{proof}[Proof of Theorem~\ref{thm:red_almst_smpl_grp}]
		If $G$ has a non-trivial abelian normal subgroup, then, by Lemma~\ref{lem:ab_prod_cse_shrt_id} and Lemma~\ref{lem:nor_subgrp}, it has a mixed identity of length at most $4$. Otherwise, $G$ is as in Theorem~\ref{thm:fitting}. Again by Lemma~\ref{lem:ab_prod_cse_shrt_id}, if $k\geq2$ or some $l_i\geq 2$ ($i\in\set{1,\ldots,k}$), then the socle of $G$ satisfies a mixed identity of length $4$, as it is a non-trivial direct product. By Lemma~\ref{lem:nor_subgrp}, the group $G$ satisfies a mixed identity of length at most $8$. 
		
		Otherwise $G$ is almost simple, with socle $S$ (a non-abelian finite simple group). If $S$ is alternating, then, by the argument on page~\pageref{fct:bd_ids_alt_grps} and Lemma~\ref{lem:nor_subgrp}, the group $G$ has a mixed identity of length at most $120$. 
		If $S$ is sporadic, then, by Lemma~\ref{lem:nor_subgrp}, the group $G$ has a mixed identity of bounded length. The cases of the symplectic and odd-dimensional orthogonal 
		groups are deferred, to Sections~\ref{SympSect} and~\ref{OddOrthSect}, respectively. 
	\end{proof}
	
	Theorem~\ref{thm:red_almst_smpl_grp} gives now the following optimal improvement of Corollary~\ref{cor:shrt_id_one_var}:
	
	\begin{corollary} 
		Every finite group $G$ has a mixed identity of length $O(|G|^{1/3}).$
	\end{corollary}
	
	\begin{proof}
		This is a consequence of Theorem~\ref{thm:red_almst_smpl_grp}, the reduction to simple groups by Lemma~\ref{lem:nor_subgrp2}, and the bounds obtained for simple groups in \cite{bradford2019lie}*{Theorem 1.1}. 
	\end{proof}
	
	
	
	\section{The projective special linear groups \texorpdfstring{$\PSL_n(q)$}{PSLn}}
	
	In this section, we prove Theorem~\ref{thm:psl}. We start with the construction of a mixed identity for $\PSL_n(q)$.
	
	\begin{lemma}\label{lem:psl_up_bd}
		There is a mixed identity of length $O(q)$ for $\PSL_n(q)$.
	\end{lemma}
	\begin{proof}
		
		For $h$ a rank-one matrix that squares to $0_V$, set $k\coloneqq 1_V+h\in\SL_n(q)$, where $V\cong\finfield_q^n$ is the natural module of $\SL_n(q)$. Then $k$ fixes the hyperplane $H\coloneqq\ker(h)$ pointwise. 
		
		Now note that $v(x,y)=[[[x,y^p],y^{-(q-1)}],y^{q+1}]\in\freegrp_2=\gensubgrp{x,y}$ is a law of length $O(q)$ in the variables $x,y$ for $\SL_2(q)$. This is, since any element $g\in\SL_2(q)$ satisfies either $g^{q-1}=\id$ if it is diagonalizable, $g^{q+1}=\id$ if it has no eigenvectors, or $g^p=\pm\id$ if it is a plus or minus a unipotent.  Let $g\in\SL_n(q)$ be arbitrary and consider the elements $k$ and $k^g$. Now, $k$ and $k^g$ fix the codimension-two subspace $U\coloneqq H\cap H.g$ and can be written in the form
		$$
		k\text{ resp. } k^g=
		\begin{pmatrix}
			\ast & \ast\\
			0 & 1_U
		\end{pmatrix}.
		$$
		Hence if we consider the matrices $v(k,k^g)$ and $v(k^g,k)$, we get
		$$
		v(k,k^g)  \text{ resp. } v(k^g,k)=\begin{pmatrix}
			1_W & \ast\\
			0 & 1_U
		\end{pmatrix}.
		$$
		where $W$ is a complement of $U$ in $V$. Hence both lie in an abelian subgroup of unipotent elements and $[v(k,k^g),v(k^g,k)]$ is trivial for all choices of $g$. Thus $w=[v(k,k^x),v(k^x,k)]\in\SL_n(q)\ast\gensubgrp{x}$ is a mixed identity for $\SL_n(q)$ and hence descends to a mixed identity for $\PSL_n(q)$ of length $O(q)$\label{result:bd_ids_psl_nq}.
	\end{proof}
	
	This shows that any almost simple group with socle $\PSL_n(q)$ has a shortest mixed identity of length $O(q)$ by Lemma~\ref{lem:nor_subgrp}.
	
	Now we prove the lower bound in Theorem~\ref{thm:psl}. The idea of the proof which we present stems from \cite{golubchikmikhalev1982generalized}. For the sake of clarity, let us first focus on the case where $n=2$ and no field automorphisms are involved, i.e.\ $f=e$. Write $\overline{\bullet}\colon\GL_2(q)\to\PGL_2(q)$ for the natural map. Let 
	$$
	w=c_0 x^{\varepsilon(1)} c_1 \cdots c_{l-1} x^{\varepsilon(l)} c_l\in\GL_2(q)\ast\gensubgrp{x}
	$$ 
	have only one variable $x$, as we may assume by Lemma~\ref{lem:red_one_var_cs}. For $c\in\GL_2(q)$ write $\fix(\overline{c})$ for the set of fixed points $p\in\proj(V)$ of $\overline{c}$. Here $\proj(V)$ denotes the projective line obtained from $V\cong\finfield_q^2$. We have the following:
	
	\begin{lemma}\label{lem:ex_good_rk_one_op}
		The following are equivalent:
		\begin{enumerate}[\normalfont(i)]
			\item $\bigcup_{j=1}^{l-1}{\fix(\overline{c}_j)}\neq\proj(V)$;
			\item There exists a linear operator $h\colon V\to V$ of rank one such that (a) $h^2=0_V$; and (b) $hc_jh\neq 0_V$ for all $j=1,\ldots,l-1$.
			\item There exists a linear operator $h\colon V\to V$ of rank one such that (a) $h^2=0_V$; and (b') $h c_1 h\cdots h c_{l-1} h\neq 0_V$.
		\end{enumerate}
	\end{lemma}
	
	\begin{proof}
		(i)$\Rightarrow$(iii): Given $v\in V$ such that $\gensubsp{v}$ is not a fixed point of any of the $\overline{c}_j$ ($j=1,\ldots,l-1$) extend it to a basis $B$ of $V$ and define $h$ by $b\mapsto v$ for $b\in B\setminus\set{v}$ and $v\mapsto 0$.
		(iii)$\Rightarrow$(ii) is obvious. Now,	conversely, that is (ii)$\Rightarrow$(i), if (a) and (b) is satisfied, then $\im(h)=\ker(h)$ is a one-dimensional subspace of $V$ which cannot be fixed by any $\overline{c}_j$ ($j\in\set{1,\ldots,l-1}$), otherwise $h c_j h=0_V$.
	\end{proof}
	
	Note that, with respect to the basis $e_1=v,e_2=b$, the linear map $h$ has the matrix
	$$
	h=
	\begin{pmatrix}
		0 & 0\\
		1 & 0
	\end{pmatrix}.
	$$
	
	\begin{lemma}\label{lem:poly_mthd}
		Assume the condition in Lemma~\ref{lem:ex_good_rk_one_op} is satisfied. If $0<l<q$, then $\overline{w}$ is non-constant on $\PSL_2(q)$.
	\end{lemma}
	
	Before we prove Lemma~\ref{lem:poly_mthd}, we need an auxiliary fact:
	
	\begin{lemma}\label{lem:poly_non_const}
		Let $V$ be a finite-dimensional $\finfield_q$-vector space. Assume $0<l<q$ and consider the map $v\colon \finfield_q\to V$ given by $v(\lambda)=v_0+\lambda v_1+\cdots+\lambda^l v_l$ for $v_0,\ldots,v_l\in V$, $v_l\neq 0$ and one of $v_0,\ldots,v_{l-1}$ linearly independent from $v_l$. Then $\im(v)$ is not contained in a one-dimensional subspace of $V$.
	\end{lemma}
	
	\begin{proof}
		Let $v_0',\ldots,v_n'$ be a basis of $V$. Rewrite $v$ as $v(\lambda)=p_0(\lambda)v_0'+p_1(\lambda)v_1'+\cdots+p_n(\lambda)v_n'$ for polynomials $p_i\in\finfield_q[X]$ ($i=0,\ldots,n$), where $v_0'\coloneqq v_l$ and $v_1'\coloneqq v_j$, where $j$ is chosen such that $v_l$ and $v_j$ are linearly independent ($j\in\set{0,\ldots,l-1}$). Then $p_0$ has degree $l$ and its coefficient of $\lambda^l$ is one and the coefficient of $\lambda^j$ is zero. Similarly, the coefficient in $p_1$ of $\lambda^j$ is one and the coefficient of $\lambda^l$ is zero.
		Choose $\mu\in\finfield_q$ such that $p_1(\mu)\neq 0$, which is possible, since $p_1$ is of degree less than $q$ and non-zero. Then $p_0(\lambda)p_1(\mu)=p_0(\mu)p_1(\lambda)$ cannot hold for all $\lambda$, since $p_0(\lambda)p_1(\mu)-p_0(\mu)p_1(\lambda)$ is a non-zero polynomial of degree $l<q$ in $\lambda$. Hence $v(\mu)$ cannot be a multiple of $v(\lambda)$. The proof is complete.
	\end{proof}
	
	\begin{proof}[Proof of Lemma~\ref{lem:poly_mthd}]
		The word $\overline{w}$ is constant if and only if $\overline{w}'$ is constant, where $w'=x^{\varepsilon(1)} c_1 \cdots c_{l-1} x^{\varepsilon(l)}$. Let $h$ be as in Lemma~\ref{lem:ex_good_rk_one_op}, then we can plug in $k(\lambda)=1_V+\lambda h$ into $w'$.
		Note that $k(\lambda)^{-1}=(1_V+\lambda h)^{-1}=1_V-\lambda h=k(-\lambda)$ as $h^2=0$. Thus one obtains
		$$
		w'(k(\lambda))=w'(1_V+\lambda h)=p_0+\lambda p_1+\cdots+\lambda^l p_l,
		$$
		where $p_0,\ldots,p_l\in\End(V)$, $p_0=c_1\cdots c_{l-1}$, and $p_l=\pm hc_1 h\cdots h c_{l-1}h=\beta h$ for some $\beta\in\finfield_q^\times$. Thus $p_0$ and $p_l$ are linearly independent as the former is invertible and the latter has rank one. Hence we may apply Lemma~\ref{lem:poly_non_const} to deduce that the image of $\lambda\mapsto w'(k(\lambda))=w'(1_V+\lambda h)$ is not contained in a one-dimensional subspace of $\End(V)=\M_2(q)$, so $\overline{w}'$ is non-constant. The proof is finished.
	\end{proof}
	
	We finish the proof of the lower bound for $\PSL_2(q)$ in Theorem~\ref{thm:psl} by proving the following lemma.
	
	\begin{lemma}\label{lem:psl_2_fnl_lem}
		Let $w\in\GL_2(q)\ast\gensubgrp{x}$ be of length $0<l\leq\frac{q}{2}+1$ such that $\overline{w}\in\PGL_2(q)\ast\gensubgrp{x}$ is of positive length. Then $\overline{w}$ is non-constant on $\PSL_2(q)$.
	\end{lemma}
	
	\begin{proof}
		As before, we apply Lemma~\ref{lem:red_one_var_cs} to get that all intermediate constants $c_j$ ($j=1,\ldots,l-1$) are non-central, i.e.\ not equal to $\lambda 1_V$ for some $\lambda\in\finfield_q^\times$. Then we can pass from 
		$$
		w=c_0 x^{\varepsilon(1)} c_1 \cdots c_{l-1} x^{\varepsilon(l)} c_l \quad \mbox{to} \quad 
		w'=x^{\varepsilon(1)} c_1 \cdots c_{l-1} x^{\varepsilon(l)}.
		$$ 
		The condition in Lemma~\ref{lem:ex_good_rk_one_op} is satisfied since each $c_j$ ($j=1,\ldots,l-1$) is non-central, so $\overline{c}_j$ has at most two fixed points. Thus 
		$$
		\card{\bigcup_{j=1}^{l-1}{\fix(\overline{c}_j)}}\leq2(l-1)\leq q<\card{\proj(V)}=q+1.
		$$ 
		Also $0<l<q$ for $q>2$, so the condition in Lemma~\ref{lem:poly_mthd} is satisfied and $\overline{w}$ is non-constant. For $q=2$ we can apply Lemma~\ref{lem:no_shrt_ids}. The proof is complete.
	\end{proof}
	
	Finally, note that the above proof for $\PSL_2(q)$ can be adapted for $\PSL_n(q)$ for $n\geq 3$:
	
	\begin{lemma}\label{lem:psl_n_fnl_lem}
		Let $n\geq 3$ and $w\in\GL_n(q)\ast\gensubgrp{x}$ be of length $0<l\leq q-1$ such that $\overline{w}\in\PGL_n(q)\ast\gensubgrp{x}$ is of positive length. Then $\overline{w}$ is non-constant on $\PSL_n(q)$.
	\end{lemma}
	
	\begin{proof}
		Let           
		$
		w'=x^{\varepsilon(1)} c_1 \cdots c_{l-1} x^{\varepsilon(l)}\in\GL_n(q)\ast\gensubgrp{x},
		$ 
		all $c_j$ non-central ($j=1,\ldots,l-1$), $n\geq 3$, and $l\leq q-1$. Moreover, $V\cong \finfield_q^n$. Then we need to find $h\in\End(V)$ such that $h^2=0$ and there is $H\leq V$ a hyperplane such that $\ker(h)=H$ and $\im(h)=\gensubsp{v}\leq H$ and $hc_jh\neq 0$ for all $j=1,\ldots,l-1$. This means $v.c_j\notin H$. Then $\gensubsp{v}\notin\fix(\overline{c}_j)$ for all $j=1,\ldots,l-1$. As each $c_j$ is non-central, $\overline{c}_j$ has at most $f=\frac{q^{n-1}-1}{q-1}+1$ fixed points in $\proj(V)$ (coming from the eigenspaces of $c_j$). These points are excluded for the choice of $\gensubsp{v}$, but 
		$$
		f(l-1)<f(q-1)=q^{n-1}-1+q-1<\card{\proj(V)}=\frac{q^n-1}{q-1}=q^{n-1}+\cdots+q+1
		$$ 
		(as $n\geq 3$), so we can choose $\gensubsp{v}$ to be a non-fixed point of all $\overline{c}_j$ ($j=1,\ldots,l-1$). Now we have to choose $H\leq V$ a hyperplane such that $v.c_j\notin H$ for all $j=1,\ldots,l-1$ and $v\in H$. This condition excludes $g=\frac{q^{n-2}-1}{q-1}$ hyperplanes containing $v$ and $v.c_j$. But 
		$$
		g(l-1)<g(q-1)=q^{n-2}-1<\frac{q^{n-1}-1}{q-1}=q^{n-2}+q^{n-3}+\cdots q+1
		$$ 
		and there are that many hyperplanes containing $v$. So we can choose a suitable hyperplane $H$. Now we have defined $h$ up to a scalar factor. 
		We can proceed as in the proof of Lemma~\ref{lem:poly_mthd} to see that $\lambda\mapsto\overline{w}'(\overline{k(\lambda)})=\overline{w}'(\overline{1_V+\lambda h})$ is non-constant. This shows that, when $\overline{w}$ is a mixed identity for $\PSL_n(q)$ with constants in $\PGL_n(q)$ and $n\geq3$, then we cannot have $\norm{w}<q$. This finishes the proof.
	\end{proof}
	
	We will now start to take field automorphisms into account and study the groups $\PSL_n(q)\rtimes\gensubgrp{\alpha\mapsto\alpha^{p^f}}$. Again, for simplicity, we start with the case $n=2.$
	
	\begin{lemma}
		The group $\PSL_2(q)\rtimes\gensubgrp{\alpha\mapsto\alpha^{p^f}}$ has a mixed identity of length $O(\frac{e}{f}p^f).$
	\end{lemma}
	
	\begin{proof}
		Let $q=p^e$ and $r=p^f$ be powers of the prime $p$ for $f\mid e$. Write $F$ for the $r$-Frobenius map on $\SL_2(q)$; $\alpha\mapsto\alpha^r$ entry-wise. For $g\in\SL_2(q)$ we have that $h\coloneqq gg^F\cdots g^{F^{e/f-1}}$ is mapped to $h^F=g^F\cdots g^{F^{e/f}}=h^g$ under $F$. But the eigenvalues $\lambda_1,\lambda_2$ of $h$ are mapped to $\lambda_1^F,\lambda_2^F$ (where $F$ is extended in the obvious way to $\overline{\finfield}_q$), so that we must have $\set{\lambda_1,\lambda_2}=\set{\lambda_1^F,\lambda_2^F}$ for any choice of a continuation of $F$. Thus either $\lambda_i^F=\lambda_i^r=\lambda_i$ for $i=1,2$, i.e.\ $\lambda_i\in\finfield_r^\times$, or $\lambda_1^F=\lambda_1^r=\lambda_2$ and $\lambda_1\lambda_2=\det(h)=1$, so that $\lambda_1\lambda_1^F=\lambda_1^{r+1}=1$. Hence, one of $h^p$, $h^{r-1}$, $h^{r+1}$ is central in $\SL_2(q)$, so that $w(x x^F\cdots x^{F^{e/f-1}})$, where $w(x)=[[[c,x^p],x^{r-1}],x^{r+1}]$ ($c$ non-central in $\SL_2(q)$), is a mixed identity for $\SL_2(q)$ (with constants in $\SL_2(q)\rtimes\gensubgrp{\alpha\mapsto\alpha^r}$) of length $e\norm{w}$/f and $\norm{w}\leq2(2(2p+r-1)+r+1)\leq 14r$. Thus we have a mixed identity of length at most $14\frac{e}{f}p^f$ for $\PSL_2(q)$ with constants in $\PSL_2(q)\rtimes\gensubgrp{\alpha\mapsto\alpha^r}$ as desired.
	\end{proof}
	
	The lower bound for $\PSL_2(q)\rtimes\gensubgrp{\alpha\mapsto\alpha^{p^f}}$ needs some additional arguments. At first we need an auxiliary lemma about the fixed points of semi-linear maps:
	
	\begin{lemma}\label{lem:fxd_pts_semilin_map}
		Let $p$ be a prime and $q=p^e$. Let $c$ be an invertible $(x\mapsto x^{p^m})$-semi-linear map ($1\leq m\leq e$) on $V\cong\finfield_q^n$. Then $\overline{c}\colon\proj(V)\to\proj(V)$ has at most $\frac{p^{m' n}-1}{p^{m'}-1}$ fixed points, where $m'\coloneqq\gcd(e,m)$.
	\end{lemma}
	
	\begin{proof}
		Any proper power of $\overline{c}$ has at least the fixed points of $\overline{c}$ on $\proj(V)$ as its fixed points. Hence we can pass to a proper power of $c$ which is $(x\mapsto x^{p^{m'}})$-semi-linear (recall that $m'=\gcd(e,m)$). Take this as our new $c$. Let $U$ be the subspace spanned by all eigenvectors of $c$. Let $e_1,\ldots,e_l$ be a basis of $U$ consisting of eigenvectors of $c$ (by assumption $l\leq n$). Then
		$$
		u.c=\begin{pmatrix}
			u_1^{p^{m'}} & \cdots & u_l^{p^{m'}}
		\end{pmatrix}
		\diag(\lambda_1,\ldots,\lambda_l)
		$$
		when it is written in the basis $e_1,\ldots,e_l$ for suitable $\lambda_i$ ($i=1,\ldots,l$). Assume there is an eigenvector $u\in U$ that has exactly $j$ coordinates unequal to zero. W.l.o.g.\ assume that these are the first $j$ coordinates.
		Then
		$$
		u.c=\begin{pmatrix} 
			\lambda_1 u_1^{p^{m'}} & \cdots & \lambda_j u_j^{p^{m'}} & 0 & \cdots & 0
		\end{pmatrix}
		=
		\begin{pmatrix}
			\lambda u_1 & \cdots & \lambda u_j & 0 & \cdots & 0    
		\end{pmatrix}
		$$
		So $u_i^{p^{m'}-1}=\lambda/\lambda_i$ ($i=1\ldots,j$). Then any other eigenvector with exactly these coordinates unequal to zero is of the form
		$$
		\begin{pmatrix}
			\alpha_1 u_1 & \cdots & \alpha_j u_j & 0 & \cdots & 0
		\end{pmatrix}
		$$
		for $\alpha_i\in\finfield_q^\times$ such that $\alpha_i^{p^{m'}-1}=\mu$ for some $\mu\in\finfield_q^\times$. The new eigenvector we get is a scalar multiple of
		$$
		\begin{pmatrix}
			u_1 & \frac{\alpha_2}{\alpha_1}u_2 & \cdots & \frac{\alpha_j}{\alpha_1} u_j & 0 & \cdots & 0
		\end{pmatrix}
		$$
		and we can choose the ratios $\alpha_i/\alpha_1$ arbitrarily such that $(\alpha_i/\alpha_1)^{p^{m'}-1}=\mu/\mu=1$, or written differently, $\alpha_i/\alpha_1\in\finfield_{p^{m'}}^\times$ ($i=2,\ldots,j$). Hence there are at most $(p^{m'}-1)^{j-1}$ fixed points of $\overline{c}$ with the first $j$ coordinates non-zero. Thus counting all fixed points by choosing any $1\leq j\leq l$ coordinates to be non-zero, we get at most
		\begin{align*}
			\sum_{j=1}^l{\binom{l}{j}\left(p^{m'}-1\right)^{j-1}} 
			& = \frac{1}{p^{m'}-1}\left(\sum_{j=0}^l{\binom{l}{j}\left(p^{m'}-1\right)^j}-1\right)\\
			& =\frac{p^{m' l}-1}{p^{m'}-1}\leq\frac{p^{m' n}-1}{p^{m'}-1}
		\end{align*}
		as desired. The proof is complete.
	\end{proof}
	
	Now we prove the lower bound in Theorem~\ref{thm:psl} for $\PSL_2(q)\rtimes\gensubgrp{\alpha\mapsto\alpha^{p^f}}$:
	
	\begin{lemma}\label{lem:lwr_bd_aut_ids_psl_2}
		Let $p$ be a prime and $q=p^e$. Let $w\in(\GL_2(q)\rtimes\gensubgrp{\alpha\mapsto\alpha^{p^f}})\ast\gensubgrp{x}$ be of length $0<l\leq\frac{e}{2f}(p^f-1)+1$ (with $f\mid e$ and $1\leq f\leq e/2$) such that $\overline{w}\in(\PGL_2(q)\rtimes\gensubgrp{\alpha\mapsto\alpha^{p^f}})\ast\gensubgrp{x}$ has positive length. Then $\overline{w}$ is not a mixed identity for $\PSL_2(q)$ with constants in $\PGL_2(q)\rtimes\gensubgrp{\alpha\mapsto\alpha^{p^f}}$.
	\end{lemma}
	
	\begin{proof}
		Set $F$ to be the Frobenius automorphism $\finfield_q\to\finfield_q; \alpha\mapsto\alpha^{p^f}$ and let $F$ act coordinate-wise on $V\cong\finfield_q^2$ and on the matrices $\M_2(q)\cong\End(V)$. Assume
		\begin{align*}
			w &= c_0 x^{\varepsilon(1)} c_1\cdots c_{l-1} x^{\varepsilon(l)} c_l\\
			&=b_0.F^{m(0)} x^{\varepsilon(1)} b_1.F^{m(1)}\cdots b_{l-1}.F^{m(l-1)} x^{\varepsilon(l)} b_l.F^{m(l)}
		\end{align*}
		is such that $\overline{w}$ is a mixed identity with $b_j\in\GL_2(q)$, integers $0\leq m(j)\leq e/f-1$ ($j=0,\ldots,l$), and maps $c_j\in\GL_2(q)\rtimes\gensubgrp{\alpha\mapsto\alpha^{p^f}}\setminus\finfield_q^\times 1_V$ ($j=1,\ldots,l-1$), as we may assume by Lemma~\ref{lem:red_one_var_cs}. Then we can shift the $F$'s so that $w=a_0 (x^{\varepsilon(1)})^{F^{n(1)}}a_1\cdots a_{l-1}(x^{\varepsilon(l)})^{F^{n(l)}}a_l$ for linear maps $a_j\in\GL_2(q)$ ($j=0,\ldots,l$) and integers $0\leq n(j)\leq e/f-1$, since $\overline{w}$ is a mixed identity. Again, we may concentrate on the word 
		$$
		w'=(x^{\varepsilon(1)})^{F^{n(1)}}a_1\cdots a_{l-1}(x^{\varepsilon(l)})^{F^{n(l)}}
		$$ 
		instead of $w$, which shall be such that $\overline{w}'\in(\PGL_2(q)\rtimes\gensubgrp{\alpha\mapsto\alpha^{p^f}})\ast\gensubgrp{x}$ is not a constant. We want to pursue the same strategy as in the proof of the lower bound for $\PSL_2(q)$. When plugging in $k(\lambda)=1_V+\lambda h$ for $x$ into $w'$, we must evaluate $k(\lambda)^{F^{n(j)}}=(1_V+\lambda h)^{F^{n(j)}}=1_V+\lambda^{p^{fn(j)}} h^{F^{n(j)}}$. We want to have that the leading coefficient $h^{F^{n(1)}}a_1 h^{F^{n(2)}}\cdots h^{F^{n(l-1)}}a_{l-1}h^{F^{n(l)}}$ of the polynomial in $\lambda$, which is obtained by evaluating $w'(k(\lambda))$, is non-zero. This means that $h^{F^{n(j)}}a_j h^{F^{n(j+1)}}$ is non-zero ($j=1,\ldots,l-1$). 
		
		Recall from the proof of the lower bound for $\PSL_2(q)$ that $h$ is a linear map such that $h\colon b\mapsto v; v\mapsto 0$ for suitable $0 \neq v \in V$ and $v \neq b \in B$, for $B$ a basis of $V$ containing $v$. The condition that $h^{F^{n(j)}}a_j h^{F^{n(j+1)}}\neq 0_V$ then means that $v.F^{n(j)}.a_j\neq\lambda v.F^{n(j+1)}$, so 
		$$
		v.F^{n(j)}.a_j.F^{-n(j+1)}=v.F^{n(j)-n(j+1)}.a_j^{F^{-n(j+1)}}\neq\mu v,
		$$
		where $\mu=\lambda.F^{-n(j+1)}$. If $n(j)=n(j+1)$, $F^{n(j)-n(j+1)}.a_j^{F^{-n(j+1)}}=a_j^{F^{-n(j+1)}}$ is a non-trivial linear map and so has at most two fixed points. Else $F^{n(j)-n(j+1)}.a_j^{F^{-n(j+1)}}$ is a $(x\mapsto x^{p^{f m(j)}})$-semi-linear map and hence has at most 
		$$
		p^{\gcd(e,f m(j))}+1\leq p^{\gcd(e,f(e/f-1))}+1\leq p^{e/2}+1
		$$ 
		fixed points by Lemma~\ref{lem:fxd_pts_semilin_map}. But $2<p^{e/2}+1$ (as $e\geq1$, $p\geq2$) and 
		$$
		(p^{e/2}+1)(l-1)\leq (p^{e/2}+1)\frac{e}{2f}(p^f-1)\leq q-1<q+1=\card{\proj(V)}
		$$ 
		as $1\leq f\leq e/2$, where we use the assumption $l\leq\frac{e}{2f}(p^f-1)+1$ in the first inequality. The second inequality holds since $\frac{p^f-1}{p^{e/2}-1}\leq\frac{2f}{e}$ as both sides evaluate to $0$ for $f=0$ and to $1$ for $f=e/2$ and the left hand side is convex as a function of $f$. Hence there is a solution $v$ to the inequalities $v.F^{n(j)}.a_j\neq\lambda v.F^{n(j+1)}$ ($j=1,\ldots,l-1$).
		
		By evaluating $w'(k(\lambda))$ we obtain a polynomial of degree $\sum_{j=1}^l{p^{fn(j)}}$ in $\lambda$. Now, if we consider all the words $w',w'^F,\ldots,w'^{F^{e/f-1}}$ we note that the sum of the degrees of the corresponding polynomials is 
		$$
		l\sum_{i=0}^{e/f-1}{p^{fi}}.
		$$ 
		Hence there is a word $w'^{F^i}$ for $0 \leq i \leq e/f-1$ that gives a polynomial of degree at most 
		$$
		l\frac{f}{e}\sum_{i=0}^{e/f-1}{p^{fi}}=l\frac{f}{e}\frac{p^{fe/f}-1}{p^f-1}=l\frac{f}{e}\frac{q-1}{p^f-1}.$$ 
		However, this is less than or equal to $q-1$ when $l\leq \frac{e}{f}(p^f-1)$. But by assumption $l\leq\frac{e}{2f}(p^f-1)+1\leq\frac{e}{f}(p^f-1)$ since $f\leq e/2$. Hence the image of the polynomial $\lambda\mapsto w'(k(\lambda))$ is not contained in a one-dimensional subspace by Lemma~\ref{lem:poly_non_const}. This completes the proof.
	\end{proof}
	
	Now we turn to the proof for $\PSL_n(q)$:
	
	\begin{lemma}\label{lem:lwr_bd_aut_ids_psl_n}
		Let $p$ be a prime and $q=p^e$. Let $w\in(\GL_n(q)\rtimes\gensubgrp{\alpha\mapsto\alpha^{p^f}})\ast\gensubgrp{x}$ be of length $0<l\leq\frac{e}{2f}(p^f-1)+1$ when $n=2$, and $0<l\leq\frac{e}{f}(p^f-1)$ for $n\geq3$ (where $f\mid e$ and $1\leq f\leq e/2$) such that $\overline{w}\in(\PGL_n(q)\rtimes\gensubgrp{\alpha\mapsto\alpha^{p^f}})\ast\gensubgrp{x}$ is of positive length. Then $\overline{w}$ is not a mixed identity for $\PSL_n(q)$ with constants in $\PGL_n(q)\rtimes\gensubgrp{\alpha\mapsto\alpha^{p^f}}$.
	\end{lemma}
	
	\begin{proof}
		As we did above, we want to plug in $k(\lambda)=1_V+\lambda h$ into $w'=x^{\varepsilon(1)}c_1\cdots c_{l-1}x^{\varepsilon(l)}\in(\GL_n(q)\rtimes\gensubgrp{\alpha\mapsto\alpha^{p^f}})\ast\gensubgrp{x}$ for a suitable rank-one operator $h\in\End(V)\cong\M_n(q)$ and a scalar $\lambda\in\finfield_q$ to get a non-constant polynomial in $\lambda$ of degree less than $q$ forcing $\overline{w}'$ to be non-constant. For this purpose we want to find a vector $v$ and a hyperplane $H$ such that $h\colon b\mapsto v$ for some $b\notin H$, $v\in H=\ker(h)$, and $hc_j h\neq 0_V$ ($j=1,\ldots,l-1$). Then $h^2=0$. By Lemma~\ref{lem:fxd_pts_semilin_map} applied to the $(x\mapsto x^{p^{f m(j)}})$-semi-linear map $c=c_j$ ($j=1\ldots,l$) we obtain that $\overline{c}_j$ has at most 
		$$\frac{p^{n\gcd(e,f m(j))}-1}{p^{\gcd(e,f m(j))}-1}\leq \frac{p^{\frac{en}{2}}-1}{p^{e/2}-1}=\frac{q^{n/2}-1}{q^{1/2}-1}
		$$ 
		fixed points, unless $m(j)=0$ and $c_j$ is linear and so $\overline{c}_j$ has at most 
		$$
		\frac{q^{n-1}-1}{q-1}+1
		$$
		fixed points in $\proj(V)=\proj(\finfield_q^n)$. Hence $\overline{c}=\overline{c}_j$ can have at most
		$$
		\max\left(\frac{q^{n/2}-1}{q^{1/2}-1},\frac{q^{n-1}-1}{q-1}+1\right)
		$$ 
		fixed points. But multiplying both terms by $q-1$ and subtracting the left from the right, we obtain
		$$q^{n-1}-1+q-1-(q^{n/2}-1)(q^{1/2}+1)=q^{n-1}-q^{\frac{n+1}{2}}-q^{n/2}+q+q^{1/2}-1.
		$$ 
		For $n=2$, this is 
		$$
		q-q^{3/2}-q+q+q^{1/2}-1=-q^{3/2}+q+q^{1/2}-1=-(q^{1/2}-1)^2(q^{1/2}+1)<0
		$$ 
		since $q\geq2$. If $n=3$, we obtain
		$$
		q^2-q^2-q^{3/2}+q+q^{1/2}-1=-(q^{1/2}-1)^2(q^{1/2}+1)<0
		$$
		as well ($q\geq2$). In these both cases, the above maximum is $q^{1/2}+1$ resp.\ $\frac{q^{3/2}-1}{q^{1/2}-1}=q+q^{1/2}+1$. So, noting that $\frac{e}{2f}(p^f-1)\leq p^{e/2}-1=q^{1/2}-1$ as in the proof for $\PSL_2(q)$ above, since by assumption $l-1\leq\frac{e}{2f}(p^f-1)$, for $n=2$ we get
		\begin{align*}
			(q^{1/2}+1)(l-1) & \leq (q^{1/2}+1)\frac{e}{2f}(p^f-1)\\
			& \leq(q^{1/2}+1)(q^{1/2}-1)=q-1\\
			& <q+1=\card{\proj(V)},
		\end{align*}
		so we find a suitable $v$ in this case. For $n=3$, we get by the assumption $l\leq\frac{e}{f}(p^f-1)$ that
		\begin{align*}
			(q+q^{1/2}+1)(l-1) &<(q+q^{1/2}+1)\frac{e}{f}(p^f-1)\\
			& \leq(q+q^{1/2}+1)2(q^{1/2}-1)\\
			& =2(q^{3/2}+q+q^{1/2}-q-q^{1/2}-1)=2(q^{3/2}-1)\\
			& <q^2+q+1=\card{\proj(V)},
		\end{align*}
		so there is a good choice for $v$ as well in this case.
		For $n\geq4$, we have
		$$
		q^{n-1}-q^{\frac{n+1}{2}}-q^{n/2}+q+q^{1/2}-1>0
		$$
		and hence $\frac{q^{n-1}-1}{q-1}+1$ is the above maximum. Note now that $q-1>\frac{e}{f}(p^f-1)$ (as $f\leq e/2$) and by assumption $l\leq\frac{e}{f}(p^f-1)$ we have
		\begin{align*}
			\left(\frac{q^{n-1}-1}{q-1}+1\right)(l-1) 
			& <\left(\frac{q^{n-1}-1}{q-1}+1\right)\frac{e}{f}(p^f-1)\\
			& <\left(\frac{q^{n-1}-1}{q-1}+1\right)(q-1)=q^{n-1}-1+q-1\\
			& <q^{n-1}+q^{n-2}+\cdots+1=\frac{q^n-1}{q-1}
		\end{align*}
		Hence in all cases we find a suitable $v$ such that $\gensubsp{v}$ is not a fixed point of any of the $\overline{c}_j$ ($j=1\ldots,l-1$). The vectors $v_j\coloneqq v.c_j$ ($j=1,\ldots,l-1$) do not lie in $\gensubsp{v}$. For $H$ all hyperplanes are allowed such that $v\in H$ and $v_j\notin H$ ($j=1,\ldots,l-1$). Counting the hyperplanes that contain $v$ and $v_j$ for one $1\leq j\leq l-1$, we get
		\begin{align*}
			\frac{q^{n-2}-1}{q-1}(l-1) 
			& <\frac{q^{n-2}-1}{q-1}\frac{e}{f}(p^f-1)\\
			& <\frac{q^{n-2}-1}{q-1}(q-1)=q^{n-2}-1\\
			& <q^{n-2}+q^{n-3}+\cdots+q+1=\frac{q^{n-1}-1}{q-1},
		\end{align*}
		hence we can choose $H$ containing $v$ but none of the $v_j$. Now we run the argument as above for $\PSL_2(q)$. For this we need that $l\leq\frac{e}{f}(p^f-1)$ which is guaranteed by the assumptions in the cases $n=2$ and $n\geq 3$. This ends the proof.
	\end{proof}
	
	It is still beyond out current understanding how to incorporate the inverse-transpose automorphism in such an argument. Indeed, the above proof cannot work if we allow the inverse-transpose automorphism. Namely, then $((1_V+\lambda h)^{-1})^\top=(1_V-\lambda h)^\top=1_V-\lambda h^\top$. Let $h=a^\top b$ for vectors $a,b\in\finfield_q^n$ with $ba^\top=0$, i.e.\ $h^2=0$. Now suppose that $c_j\in\GL_n(q)$ is a critical constant with $v c_j v^\top=0$ for all $v\in\finfield_q^n$, i.e.\ $c_j$ is alternating. In that case $h^\top c_j h=b^\top ac_j a^\top b=0$ occurs in the product $c_0 h^\ast c_1\cdots c_{l-1} h^\ast c_l$ and the leading coefficient of $\lambda^{\sum_{j=1}^l{p^{fn(j)}}}$ in $w'(1_V+\lambda h)$ would be zero. Hence the above argument breaks down. 
	
	
	\section{An alternative approach to \texorpdfstring{$\PSL_2(q)$}{PSL2}}
	
	\begin{lemma} \label{psl2second}
		The shortest mixed identity $\overline{w}$ for $\PSL_2(q)$ is of length at least $q/8$.
	\end{lemma}
	
	Let $A \ast_C B$ be the amalgamated free product of the groups $A$ and $B$ over the the common subgroup $C$. Recall that a \emph{reduced expression} in $A \ast_C B$ is a tuple $(c;a_0,b_0,\ldots,a_l,b_l)$, where $l \geq 0$; $c\in C$; $a_j\in A \setminus C$ for $j\geq 1$ and $b_j\in B\setminus C$ for $j\leq l-1$, while $a_0\in(A\setminus C)\cup\set{1_C}$ and $b_l\in(B\setminus C)\cup\set{1_C}$. The key fact we need is that if $(c;a_0,b_0,\ldots,a_l,b_l)$ is a reduced expression, then $c a_0 b_0 \cdots a_l b_l$ is a non-trivial element of $A\ast_C B$, unless $l=0$ and $c=a_0=b_0=1_C$. 
	
	In particular, these observations apply to the free product $A\ast B$. In this case we shall also require the converse observation, namely that if $G$ is a group, generated by the two subgroups $A$ and $B$, such that for all $a_0,\ldots,a_l\in A$ and $b_0,\ldots,b_l\in B$, with all except possibly $a_0$ and $b_l$ non-trivial, $a_0 b_0\cdots a_l b_l$ is non-trivial also unless $l=0$ and $a_0=b_0=1_G$, then the natural map $A\ast B\rightarrow G$ is an isomorphism. 
	
	\begin{theorem}[\cite{serre1980trees}*{Chapter~II, Theorem~6}]\label{thm:serre_alg_prod}
		Let $$B(\finfield_q),B(\finfield_q[t])\leq\GL_2(\finfield_q[t])$$ be, respectively, the subgroups of invertible upper triangular matrices over $\finfield_q$ and $\finfield_q[t]$. Then $\GL_2(\finfield_q[t])$ is the amalgamated free product of $\GL_2(\finfield_q)$ and $B(\finfield_q[t])$ over $B(\finfield_q)=\GL_2(\finfield_q) \cap B(\finfield_q[t])$. 
	\end{theorem}
	
	\begin{lemma}\label{lem:psl_2_free_prod}
		Let
		\begin{align*}
			g & =-
			\begin{pmatrix}
				1-t^3 & t+t^2-t^4 \\ -t & 1-t^2
			\end{pmatrix}\\
			& = \begin{pmatrix}
				1 & t^2 \\ 0 & 1
			\end{pmatrix}
			\begin{pmatrix}
				0 & 1 \\ -1 & 0
			\end{pmatrix}
			\begin{pmatrix}
				1 & t \\ 0 & 1
			\end{pmatrix}
			\begin{pmatrix}
				0 & 1 \\ -1 & 0
			\end{pmatrix}
			\begin{pmatrix}
				1 & t \\ 0 & 1
			\end{pmatrix}
			\in\SL_2(\finfield_q[t])
		\end{align*}
		and let $\overline{g}$ be the image of $g$ in $\PSL_2(\finfield_q[t])$. Then $\gensubgrp{\overline{g}}\cong\ints$, and $\PSL_2(\finfield_q),\gensubgrp{\overline{g}}\leq\PSL_2(\finfield_q[t])$ generate their free product. 
	\end{lemma}
	
	\begin{proof}
		For the first claim it suffices to check that $g^n$ is non-central in $\SL_2(\finfield_q[t])$. Let: 
		$$
		u(f)=\begin{pmatrix} 1 & f \\ 0 & 1 \end{pmatrix}\text{ and }
		r=\begin{pmatrix} 0 & 1 \\ -1 & 0 \end{pmatrix}
		$$ 
		for $f\in\finfield_q[t]$, so that $g = u(t^2)ru(t)ru(t)$. Then for $n\geq 1$, 
		\begin{equation}\label{eq:g_power_eqn}
			g^n = u(t^2)\big( ru(t)ru(t+t^2)\big)^{n-1}ru(t)ru(t)
		\end{equation}
		is a reduced element of the amalgam from Theorem~\ref{thm:serre_alg_prod}, so is non-central in $\SL_2(\finfield_q[t])$. 
		
		Now let $n(j)\in\ints$, $h_j\in\PSL_2(\finfield_q)$ ($0\leq j\leq l$) be such that $\overline{w}=\overline{g}^{n(0)} h_0\cdots\overline{g}^{n(l)} h_l\in\gensubgrp{\overline{g}}\ast\PSL_2(\finfield_q)$ is a non-trivial reduced word, 
		and let $\tilde{h}_j\in\SL_2(\finfield_q)$ be a lift of $h_j$ (so that $n(j)\neq 0$ for $j\geq 1$, and $\tilde{h}_j\neq\pm 1_2$ for $j\leq l-1$). We claim that $\overline{w}$ is also a non-trivial element of $\PSL_2(\finfield_q[t])$. We have that $\overline{w}$ lifts to: 
		$$
		w=\pm g^{n(0)}\tilde{h}_0\cdots g^{n(l)}\tilde{h}_l
		$$
		in $\SL_2(\finfield_q[t])$. By Equation~\eqref{eq:g_power_eqn}, an elementary contraction to this expression for $w$, as an element of the amalgam from Theorem~\ref{thm:serre_alg_prod}, corresponds to an index $j$ such that $\tilde{h}_j\in B(\finfield_q)\cap\SL_2(\finfield_q)$. Therefore let $a_j\in\finfield_q^\times$, $b_j\in\finfield_q$ be such that:  
		$$
		\tilde{h}_j =
		\begin{pmatrix}
			a_j & b_j \\ 0 & a_j^{-1}
		\end{pmatrix}. 
		$$
		We claim that for such $j$, and for $x(t)\in\set{t,-t^2}$, $y(t)\in\set{-t,t^2}$ we have: 
		\begin{equation} \label{eq:sl_2_red_eqn}
			\pm r u(x) h_j u(y) r 
			= k_1 u (f) k_2 \text{ or } k_1, 
		\end{equation}
		for some $k_i\in\SL_2(\finfield_q)\setminus B(\finfield_q)$ and $f\in\finfield_q[t]$ non-constant. Applying all transformations~\eqref{eq:sl_2_red_eqn} to $w$ at the indices $j$ for which $\tilde{h}_j\in B(\finfield_q)\cap\SL_2(\finfield_q)$, we obtain a non-trivial reduced form for $w$ in the amalgamated free product. Thus, as an element of $\GL_2(\finfield_q[t])$, $w\neq\pm 1_2$, and $\overline{w}\in\PSL_2(\finfield_q[t])$ is non-trivial, as desired. 
		
		We now prove the claim:
		
		\emph{Case 1:} $a_j\neq\pm 1$: 
		$$
		r u(x) h_j u(y) r = 
		\begin{pmatrix}
			0 & a_j^{-1} \\ -a_j & 0
		\end{pmatrix} 
		u(f) r
		$$
		where $f(t)=a_j^{-1}b_j+a_j^{-2}x(t)+y(t)$ is non-constant, since $a_j^2\neq 1$, and either $x(t)=-y(t)$ or $x(t)$ and $y(t)$ are of different degrees. 
		
		\emph{Case 2:} $a_j=\pm 1$, $b_j\neq 0$: 
		$$
		r u(x) h_j u(y) r=\pm r u (x+y\pm b_j) r
		$$
		which is of the required form, as either $x(t)+y(t)$ is non-constant, or $x(t)=-y(t)$, in which case we have: 
		$$
		r u(x) h_j u(y) r 
		=\begin{pmatrix}
			\mp 1 & 0 \\ b_j & \mp 1
		\end{pmatrix}
		$$
		which is also of the desired form. This verifies the two cases.
	\end{proof}
	
	\begin{proof}[Proof of Lemma \ref{psl2second}]
		Let $\overline{w}\in\PSL_2(\finfield_q)\ast\gensubgrp{x}$ be a mixed identity for $\PSL_2(\finfield_q)$. By Lemma~\ref{lem:psl_2_free_prod}, there is a monomorphism $\iota\colon\PSL_2(\finfield_q)\ast\gensubgrp{x}\to\PSL_2(\finfield_q[t])$ restricting to the identity on $\PSL_2(\finfield_q)$, with $\deg(\iota(x))\leq 4$. For $\alpha\in\finfield_q$, let $\pi_{\alpha}\colon\PSL_2(\finfield_q[t])\to\PSL_2(\finfield_q)$ be the epimorphism induced by evaluation of $t$ at $\alpha$ (equivalently, the congruence homomorphism modulo $t-\alpha$). Then $(\pi_{\alpha}\circ\iota)(\overline{w})=\overline{1}_2$ for all $\alpha\in\finfield_q$. Let $W\in\SL_2(\finfield_q[t])$ be a lift of $\iota(\overline{w})$. At least one of the polynomials: $W_{11}(t),W_{12}(t),W_{21}(t),W_{22}(t)\in\finfield_q[t]$ is non-constant, and every $\alpha\in\finfield_q$ is a solution to one of the two systems of equations: 
		$$
		\begin{pmatrix}
			W_{11}(t) & W_{12}(t) \\ 
			W_{21}(t) & W_{22}(t)
		\end{pmatrix} 
		= \begin{pmatrix}
			1 & 0 \\ 0 & 1
		\end{pmatrix}  \text{ or }
		\begin{pmatrix}
			-1 & 0 \\ 0 & -1
		\end{pmatrix}\text{.}
		$$
		Meanwhile, the $W_{ij}(t)$ have degree at most $4l$. Hence 
		$$
		\begin{pmatrix}
			W_{11}(t)^2 & W_{12}(t)^2 \\ 
			W_{21}(t)^2 & W_{22}(t)^2
		\end{pmatrix} 
		= \begin{pmatrix}
			1 & 0 \\ 0 & 1
		\end{pmatrix}.
		$$
		Thus $8l \geq q$, as the $W_{ij}(t)^2$ have degree at most $8l$. 
	\end{proof}
	
	
	
	
	We end this Section with a conjecture.
	
	\begin{conj} \label{SLnPolyFreeProdConj}
		There exists an absolute constant $C>0$ such that for any field $\mathbb{F}$ and every $n \geq 2$, there exists $g \in \SL_n (\mathbb{F}[t])$, the entries of which are polynomials of degree at most $C$, such that the image $\overline{g}$ of $g$ in $\PSL_n (\mathbb{F}[t])$ has infinite order and $\langle \overline{g} \rangle , \PSL_n (\mathbb{F}) \leq \PSL_n (\mathbb{F}[t])$ generate their free product. 
	\end{conj}
	
	If Conjecture~\ref{SLnPolyFreeProdConj} is true, then the lower bound in Theorem \ref{thm:psl} for $\PSL_n(q)$ would follow by precisely the same argument as we have given for $\PSL_2(q)$ above. By the results of Stepanov~\cite{stepanov2010about}, there \emph{does} exist an element $g$ as above for every $\mathbb{F}$ and $n \geq 2$, but without the uniform bound on the degrees of the elements.
	
	\section{The projective symplectic groups \texorpdfstring{$\PSp_{2m}(q)$}{PSp2mq}} \label{SympSect}
	
	Surprisingly, in contrast to the projective general linear case, there are mixed identities of bounded length for the symplectic groups $\PSp_{2m}(q)$ for $m\geq 2$. (Note that for $m=1$ we have $\PSp_2(q)\cong\PSL_2(q)$ so there are no short identities by Theorem~\ref{thm:psl}.) This is a theorem due to Tomanov~\cite{tomanov1985generalized} in odd characteristic. For the sake of clarity, we reprove it here briefly and also establish the case when $q$ is even, i.e.\ $\finfield_q$ is of characteristic two.
	
	\subsection{Tomanov's result for $\PSp_{2m}(q)$ for $m\geq 2$}
	
	\begin{theorem}[Tomanov] The group $\PSp_{2m}(q)$ for $m\geq 2$ satisfies a mixed identity of length $8.$
	\end{theorem}
	
	Let $R$ be a commutative ring of characteristic $\neq 2$ and $m\geq 2$. Consider the symplectic group $\Sp_{2m}(R)$ consisting of those matrices in $\M_{2m}(R)$ that preserve the standard non-degenerate alternating bilinear form $f\colon R^{2m}\times R^{2m}\to R$ given by
	$f(u,v)=u\Omega v^\top$,
	where $\Omega\coloneqq\left(\begin{smallmatrix}
		0_m & 1_m \\ -1_m & 0_m
	\end{smallmatrix} \right)$.
	Then, $\Sp_{2m}(R)$ can be described concretely and a matrix $\left(\begin{smallmatrix}
		a & b \\ c & d
	\end{smallmatrix} \right)\in\M_{2m}(R)$ with $a,b,c,d \in\M_m(R)$ lies in $\Sp_{2m}(R)$ if and only if 
	$$
	a b^\top-b a^\top=0_m, \quad -b c^\top+a d^\top=1_m \quad \mbox{and} \quad c d^\top-d c^\top=0_m.
	$$
	
	Now let $g$ be an arbitrary element of $\Sp_{2m}(R)$ that satisfies $g^2=1_{2m}$. Then it follows that 
	$$
	f(v.g,v) = f(v.g^2,v.g)=f(v,v.g)=-f(v.g,v).
	$$ 
	Hence $f(v.g,v)=0$ for all $v\in R^{2m}$ since $R$ is of characteristic $\neq 2$. In particular, it follows that the $(m+1,1)$-entry of $g$ must vanish as
	$g_{m+1,1}=-f(e_1.g,e_1)=f(e_1,e_1.g)=0$. Let's fix 
	$$
	g_0\coloneqq\diag\left(\left(\begin{smallmatrix}
		0 & 1 \\ 1 & 0
	\end{smallmatrix} \right),1_{m-2},\left(\begin{smallmatrix}
		0 & 1 \\ 1 & 0
	\end{smallmatrix} \right),1_{m-2}\right)\in\Sp_{2m}(R),
	$$ 
	which satisfies $g_0^2=1_{2m}$ and is a non-scalar element of $\Sp_{2m}(R)$. It is well-defined since by assumption $m\geq 2$. We conclude that for every $x\in\Sp_{2m}(R)$ the $(m+1,1)$-matrix entry of the matrix $g=g_0^x$ vanishes and hence $e_{1,m+1}g_0^x e_{1,m+1}=0_{2m}\in\M_{2m}(R)$. 
	
	Consider now the matrix
	$k\coloneqq 1_{2m}+e_{1,m+1}\in\Sp_{2m}(R)$, which is a symplectic transvection. We claim that 
	$g_0^x k g_0^x$ and $k$ commute. Indeed,
	$$
	g_0^x k g_0^x k = (1_{2m} + g_0^x e_{1,m+1} g_0^x)(1_{2m}+e_{1,m+1})=1_{2m}+g_0^x e_{1,m+1} g_0^x + e_{1,m+1}
	$$
	and similarly for $k g_0^x k g_0^x$. Hence, we conclude that for all $x\in\Sp_{2m}(R)$, we have
	$$
	w(x)= [g_0^x k g_0^x,k]=1_{2m}.
	$$
	Now if $q$ is odd, we can directly set $R\coloneqq\finfield_q$ and $w$ becomes a mixed identity of $\Sp_{2m}(q)$ of length 8, which descends to a mixed identity $\overline{w}$ of $\PSp_{2m}(q)$ since $g_0$ and $k$ are non-central. When $\finfield_q$ is of characteristic two (i.e.\ $q$ is even) assume that there is a surjective homomorphism $\overline{\bullet}\colon\Sp_{2m}(R)\twoheadrightarrow\Sp_{2m}(q)$ induced by a homomorphism $\varphi\colon R\twoheadrightarrow\finfield_q$. Then $\overline{w}$ clearly is a mixed identity of $\Sp_{2m}(q)$ which again descends to a mixed identity for $\PSp_{2m}(q)$. It remains to define the homomorphism $\varphi$ properly. For this purpose set $R\coloneqq\ints[X]$ and let $\varphi\colon\ints[X]\twoheadrightarrow\finfield_q$ be a surjective homomorphism. Then $\overline{\bullet}$ is surjective, since the symplectic transvections are elements of the symplectic groups $\Sp_{2m}(\ints[X])$ and $\Sp_{2m}(q)$ which are mapped onto each other by $\overline{\bullet}$ and they even generate $\Sp_{2m}(q)$.
	This finishes the proof.
	
	
	\subsection{Proof of Theorem~\ref{thm:sp}}
	
	In this subsection, we will prove that any mixed identity $\overline{w}\in\PSp_{2m}\ast\freegrp_r$ for $\PSp_{2m}(q)$ of length $\leq q/2+1$ has a critical constant which lifts to an involution in $\Sp_{2m}(q)$. First we need a lemma.
	
	\begin{lemma}\label{lem:quad_elts_alt}
		Let $q$ be odd. For an element $c\in\Sp_{2m}(q)$ the following are equivalent:
		\begin{enumerate}[\normalfont(i)]
			\item $c^2=1_{2m}$, i.e.\ $c$ squares to the identity.
			\item The form $g(u,v)\coloneqq f(u.c,v)$ is alternating, where $f$ is the non-degenerate alternating form associated to $\Sp_{2m}(q)$.
		\end{enumerate}
		The implication (ii)$\Rightarrow$(i) also holds for $q$ even. In this case, it suffices that $f$ and $g$ are (skew) symmetric.
	\end{lemma}
	
	\begin{proof}
		(i)$\Rightarrow$(ii): Let $v\in V\cong\finfield_q^{2m}$ be arbitrary. Then $g(v,v)=f(v.c,v)=-f(v,v.c)=-f(v.c,v.c^2)=-f(v.c,v)=-g(v,v)=0$ as $f$ is skew-symmetric, $c$ preserves $f$, and $\finfield_q$ has odd characteristic. Thus $g$ is alternating.
		
		(ii)$\Rightarrow$(i): Assume $g$ is alternating. Then $g(u,v)=f(u.c,v)=-g(v,u)=-f(v.c,u)=-f(v,u.c^{-1})=f(u.c^{-1},v),$ holds for all $u,v\in V$, since $g$ and $f$ are skew-symmetric and $c$ preserves $f$. Hence, as $f$ is non-degenerate, $u.c=u.c^{-1}$ for all $u$, so that $c=c^{-1}$ and thus $c^2=1_{2m}$. This argument also works when $q$ is even and $f$ and $g$ are (skew) symmetric.
	\end{proof}
	
	To prove the lower bound in Theorem~\ref{thm:sp}, we define $k(\lambda)$ for $\lambda\in\finfield_q$ by $x.k(\lambda)\coloneqq x+\lambda f(x,v)v=x.(1_V+\lambda h)$ for a vector $v$ which we still have to choose and consider the expression $w'(k(\lambda))$, where 
	$$
	w=c_0x^{\varepsilon(1)}c_1\cdots c_{l-1}x^{\varepsilon(l)}c_l \quad\text{and}\quad w'=x^{\varepsilon(1)}c_1\cdots c_{l-1}x^{\varepsilon(l)}
	$$ 
	are of length $\leq q/2+1$ such that $\overline{w}\in\PSp_{2m}\ast\gensubgrp{x}$ is of positive length. This $k(\lambda)$ is a symplectic transvection for all $v\in V\setminus\set{0}$. Again, if we can choose $v$ in such a way that $hc_j h\neq 0$ for all intermediate constants $c_j$ ($j=1,\ldots,l-1$) as in Lemma~\ref{lem:ex_good_rk_one_op} and if $l<q$ (which holds for all $q>2$), then we can apply the proof of Lemma~\ref{lem:poly_mthd} to get that $\overline{w}$ is not a mixed identity for $\PSp_{2m}(q)$. (For $q=2$ there is Lemma~\ref{lem:no_shrt_ids}.) We rewrite the former condition as $x.hc_jh=f(f(x,v)v.c_j,v)v=f(x,v)f(v.c_j,v)v\neq 0$. This means that $f(v.c_j,v)\neq 0$ for all $j=1,\ldots,l-1$. We claim that we can find a suitable $v\in V$ whenever all $g_j\coloneqq f(\bullet.c_j,\bullet)$ ($j=1,\ldots,l-1$) are non-alternating. To establish this claim, we need the following lemma.
	
	\begin{lemma}\label{lem:non_alt_frm}
		Let $g\colon V\times V\cong\finfield_q^{2m}\times\finfield_q^{2m}\to\finfield_q$ be a non-alternating form. Set $V(g)\coloneqq\set{v\in V\setminus\set{0}}[g(v,v)=0]$. Then $\card{V(g)}\leq 2 q^{2m-1}-1$.
	\end{lemma}
	
	\begin{proof}
		As $g$ is not alternating, we have that the polynomial $p(v)\coloneqq g(v,v)=\sum_{i\leq j\leq 2m}{g_{ij}v_iv_j}\neq 0$ as there exists a $v$ such that $p(v)=g(v,v)\neq 0$. This expression $p(v)$ is then a non-zero polynomial in the variables $v_1,\ldots,v_{2m}$ of degree two. Hence by the Schwartz-Zippel lemma it has at most $2 q^{2m-1}$ solutions, i.e.\ $V(g)\leq 2 q^{2m-1}-1$ as the zero vector is not included in $V(g)$ but $p(0)=0$. This completes the proof.
	\end{proof}
	
	Now we can prove the following lemma.
	
	\begin{lemma}\label{lem:non_alt_frms_sol}
		Let $w=c_0x^{\varepsilon(1)}c_1\cdots c_{l-1}x^{\varepsilon(l)}c_l\in\GL_{2m}(q)\ast\gensubgrp{x}$ be of length $0<l\leq q/2+1$ such that all $g_j=f(\bullet.c_j,\bullet)$ ($j=1,\ldots,l-1$) are non-alternating. Then $\overline{w}\in\PGL_{2m}(q)\ast\gensubgrp{x}$ is non-constant on $\PSp_{2m}(q)$.
	\end{lemma}
	
	\begin{proof}
		We just have to find $v$ such that $g_j(v,v)=f(v.c_j,v)\neq 0$ for all $j=1,\ldots,l-1$. But since $l\leq q/2+1$ and $\card{V(g_j)}\leq 2 q^{2m-1}-1$ by Lemma~\ref{lem:non_alt_frm} we get that 
		$$
		\card{V\setminus\set{0}\setminus\bigcup_{j=1}^{l-1}{V(g_j)}}\geq \card{V}-1-\sum_{j=1}^{l-1}{\card{V(g_j)}}\geq q^{2m}-1-(q/2)\cdot(2 q^{2m-1}-1)>0
		$$
		for $q>2$. So there is a legal choice for $v$. Also, then $0<l\leq q/2+1<q$, so the proof of Lemma~\ref{lem:poly_mthd} applies. For $q=2$ we apply Lemma~\ref{lem:no_shrt_ids}. The proof is complete.
	\end{proof}
	
	Hence by Lemma~\ref{lem:quad_elts_alt} and~\ref{lem:non_alt_frms_sol}, we immediately obtain the following corollary.
	
	\begin{corollary}\label{cor:non_quad_elts_sol}
		Let $w=c_0x^{\varepsilon(1)}c_1\cdots c_{l-1}x^{\varepsilon(l)}c_l\in\Sp_{2m}\ast\gensubgrp{x}$ be of length $0<l\leq q/2+1$ such that $c_j^2\neq 1_{2m}$ for all $j=1,\ldots,l-1$. Then $\overline{w}\in\PSp_{2m}(q)\ast\gensubgrp{x}$ is non-constant on $\PSp_{2m}(q)$.
	\end{corollary}
	
	\begin{proof}[Proof of the lower bound in Theorem~\ref{thm:sp}]
		We have to show that, if $\overline{w}\in\PSp_{2m}(q)\ast\freegrp_r$ (which now has the free variables $x_1,\ldots,x_r$) has no \emph{critical} constants that lift to involutions in $\Sp_{2m}(q)$, then still, if it is a mixed identity for $\PSp_{2m}(q)$, it must have length $>q/2+1$.
		Indeed, non-critical constants may lift to involutions and still the mixed identity $\overline{w}$ for $\PSp_{2m}(q)$ must have length bigger than $q/2+1$. More concretely, write 
		$$
		w=c_0x_{i(1)}^{\varepsilon(1)}c_1\cdots c_{l-1} x_{i(l)}^{\varepsilon(l)}c_l
		$$ 
		and assume that $\overline{w}\in\PSp_{2m}(q)\ast\freegrp_r$ is constant on $\PSp_{2m}(q)$; $c_j\in\Sp_{2m}(q)$ ($j=0,\ldots,l$). We proceed as in the proof of Lemma~\ref{lem:red_one_var_cs}. Let $s\colon x_i\mapsto g_{-i} x g_i$ for some tuple $(g_{\pm i})_{i=1}^r\in\Sp_{2m}^{2r}(q)$ and consider the word 
		$$
		w'\coloneqq w(s(x_1),\ldots,s(x_r))=c_0' x^{\varepsilon(1)} c_1' \cdots c_{l-1}' x^{\varepsilon(l)} c_l'\in\Sp_{2m}(q)\ast\gensubgrp{x}.
		$$ 
		Then we have that $c_j'=g_{\varepsilon(j) i(j)}^{\varepsilon(j)}c_j g_{-\varepsilon(j+1) i(j+1)}^{\varepsilon(j+1)}$ ($j=1,\ldots,l-1$). By assumption, $c_j$ and hence $c_j'$ does not square to one when $j\in J_-(w)$ (since the latter is conjugate to the former). We have to make sure that $c_j'$ does not square to one for $j\in J_0(w)\cup J_+(w)$. This means 
		\begin{equation}\label{eq:ilgl_consts}
			g_{\varepsilon(j) i(j)}^{\varepsilon(j)}c_j g_{-\varepsilon(j+1) i(j+1)}^{\varepsilon(j+1)}\neq c
		\end{equation} 
		where $c^2=1_{2m}$. 
		
		At first, we assume that $q$ is odd. Then by \cite{fulmanguralnickstanton2017asymptotics}*{page~889}, there are precisely 
		$$
		f=\card{\Sp_{2m}(q)}\sum_{i=0}^m{\frac{1}{\card{\Sp_{2i}(q)}\card{\Sp_{2(m-i)}(q)}}}
		$$
		solutions $c\in\Sp_{2m}(q)$ to $c^2=1_{2m}$. So there are $f\card{\Sp_{2m}(q)}^{2r-1}$ solutions to the negation of the Inequalities~\eqref{eq:ilgl_consts}. 
		We can weakly estimate $f$ for our purposes. Indeed, $\card{\Sp_{2m}(q)}=q^{m^2}\prod_{i=1}^m{(q^{2i}-1)}\geq q^{m^2}\prod_{i=1}^m{q^{2i-1}}\geq q^{2m^2}$ for $m\geq0$, so we have
		$$
		f\leq\card{\Sp_{2m}(q)}\sum_{i=0}^m{\frac{1}{q^{2i^2}\cdot q^{2(m-i)^2}}}\leq\card{\Sp_{2m}(q)}\sum_{i=0}^m{\frac{1}{q^{(i+m-i)^2}}}
		=\card{\Sp_{2m}(q)}\cdot\frac{m+1}{q^{m^2}},
		$$
		where we use the arithmetic-geometric mean inequality $\left(\frac{a+b}{2}\right)^2\leq\frac{a^2+b^2}{2}$.
		So if $(l-1)f\card{\Sp_{2m}(q)}^{2r-1}<\card{\Sp_{2m}(q)}^{2r}$ by counting we are done, as then $w'\in\Sp_{2m}(q)$ has no intermediate constants that square to one. This is equivalent to
		$l-1<\card{\Sp_{2m}(q)}/f$. But we have that $l-1<q$ (as $l\leq q/2+1$ by assumption) and from the above that $\card{\Sp_{2m}(q)}/f\geq\frac{q^{m^2}}{m+1}$, so it suffices to show that $q\leq\frac{q^{m^2}}{m+1}$, i.e.\ $m+1\leq q^{m^2-1}$ which holds for all $q$ and $m\geq 2$ as desired. So $w'\in\Sp_{2m}(q)\ast\gensubgrp{x}$ has no intermediate constants that lift to involutions, since $w\in\Sp_{2m}(q)\ast\freegrp_r$ had no critical constants that lift to involutions (both $w$ and $w'$ are of the same length). Hence, by Corollary~\ref{cor:non_quad_elts_sol}, this finishes the proof for $q$ odd.
		
		For $q$ even, the number of involutions in $\Sp_{2m}(q)$ according to \cite{fulmanguralnickstanton2017asymptotics}*{page~891} is given by
		\begin{equation}\label{eq:num_inv_q_ev}
			f=\card{\Sp_{2m}(q)}\left(\sum_{\substack{i=0\\ \text{$i$ even}}}^m{1/A_i}+\sum_{\substack{i=2\\ \text{$i$ even}}}^m{1/B_i}+\sum_{\substack{i=1\\ \text{$i$ odd}}}^m{1/C_i}\right)
		\end{equation}
		where
		\begin{align*}
			A_i &= q^{i(i+1)/2+i(2m-2i)}\card{\Sp_i(q)}\card{\Sp_{2m-2i}(q)}\\
			B_i &= q^{i(i+1)/2+i(2m-2i)}q^{i-1}\card{\Sp_{i-2}}\card{\Sp_{2m-2i}(q)}\\
			C_i &= q^{i(i+1)/2+i(2m-2i)}\card{\Sp_{i-1}(q)}\card{\Sp_{2m-2i}(q)}.
		\end{align*}
		So since $\card{\Sp_{i-2}(q)}q^{i/2-1}\leq\card{\Sp_{i-1}(q)}\leq\card{\Sp_{i}(q)}$ we obtain
		\begin{align*}
			A_i,B_i,C_i &\geq q^{i(i+1)/2+i(2m-2i)}q^{i/2-1}\card{\Sp_{i-2}(q)}\card{\Sp_{2m-2i}(q)}\\
			&\geq q^{i(i+1)/2+i(2m-2i)+i/2-1+2(\frac{i-2}{2})^2 +2(m-i)^2}\\
			&= q^{\frac{1}{2}i^2+\frac{1}{2}i+2mi-2i^2+\frac{1}{2}i-1+\frac{1}{2}i^2-2i+2+2m^2-4mi+2i^2}.
		\end{align*}
		The exponent of $q$ is here $i^2-i+1-2mi+2m^2.$ For fixed $m$, this expression gets minimal when $i=m+1/2$. But in Equation~\eqref{eq:num_inv_q_ev} we have $i\leq m$, so plugging in $i=m$ gives the lower bound $m^2-m+1$. Again, we have to show that $l-1<\card{\Sp_{2m}(q)}/f$, but $l-1<q$ and, by Equation~\eqref{eq:num_inv_q_ev} and the bound we obtained for the exponent of $q$, it holds that $\card{\Sp_{2m}(q)}/f\geq\frac{q^{m^2-m+1}}{3(m/2+1)}$. Here the the expression $3(m/2+1)$ comes from the fact that Equation~\eqref{eq:num_inv_q_ev} has at most that many summands. Hence we have to show that $q\leq\frac{q^{m^2-m+1}}{3(m/2+1)}$ which means $q^{m^2-m}\geq 3(m/2+1)$. This holds for $q>2$ and $m\geq2$. For $q=2$ we apply Lemma~\ref{lem:no_shrt_ids}. Thus we are done for $q$ even as well.
	\end{proof}
	
	This finishes the first half of the proof and we are left to prove the upper bound $O(q)$ for $q$ even.
	\begin{proof}[Proof of the upper bound in Theorem~\ref{thm:sp}]
		The proof is essentially the same as the one for Lemma~\ref{lem:psl_up_bd}.
		Let 
		$$
		k\coloneqq 1_V+h\in\Sp_{2m}(q),
		$$ 
		where $x.h=f(x,v)v$ with $v\neq 0$, be a symplectic transvection. Here $V\cong\finfield_q^{2m}$ is the natural module of $\Sp_{2m}(q)$. Proceed as in the proof of Lemma~\ref{lem:psl_up_bd} to get a mixed identity $w\in\SL_{2m}(q)\ast\gensubgrp{x}$ for $\SL_{2m}(q)$ which descends to a mixed identity $\overline{w}$ of $\PSL_{2m}(q)$. But the only constants involved in $w$ are powers of $k$ which belong to $\Sp_{2m}(q)$, so that $w$ is also a mixed identity for $\Sp_{2m}(q)$ (with constants in $\Sp_{2m}(q)$) which descends to a mixed identity of $\PSp_{2m}(q)$. The problem with characteristic two is just that the map $k$ is then an involution, which was excluded by the assumptions. Thus the proof is complete, since $\overline{w}$ is of length $O(q)$ as in Lemma~\ref{lem:psl_up_bd}.
	\end{proof}
	
	\section{The odd-degree projective orthogonal groups \texorpdfstring{$\POmega^\circ_{2m-1}(q)$}{POmega2m-1q}} \label{OddOrthSect}
	
	Similarly to the symplectic groups, the orthogonal groups $\POmega^\circ_{2m-1}(q)$ ($m\geq3$ odd, or $q\equiv 1$ mod $4$) have a short mixed identity. This is also a result of Tomanov~\cite{tomanov1985generalized}. We reprove it here:
	
	\begin{theorem}[Tomanov] \label{thm:tomanov1}
		There exists a mixed identity for $\POmega^\circ_{2m-1}(q)$ for $m\geq3$ odd, or $q\equiv 1$ mod $4$ of length $16$.
	\end{theorem}
	\begin{proof}
		Consider $\POmega^\circ_{2m-1}(q)$, $m \geq 3$, and assume $q\equiv 1$ mod $4$ or that $m$ is odd. Let 
		$$
		\Omega=\begin{pmatrix} 0 & \cdots & 0 & 1\\
			\vdots & \iddots & \iddots & 0\\
			0 & \iddots & \iddots & \vdots\\
			1 & 0 & \cdots & 0
		\end{pmatrix}
		$$
		be the matrix of the symmetric bi-linear form $f$ which is stabilized by $\GO^\circ_{2m-1}(q)$. Define 
		$$
		g_0\coloneqq\diag(-1_{m-1},1,-1_{m-1})=-1_{2m-1}+2 e_{m,m}.
		$$
		We show that $g_0$ lies in $\Omega_{2m-1}^\circ(q)$ when $m$ is odd or $q\equiv 1$ modulo $4$ (i.e.\ $-1$ is a square in $\finfield_q$). When $m$ is odd, we have that $g_0$ is the product of the elements 
		$$x\coloneqq\diag(-1_{\frac{m-1}{2}},1_m,-1_{\frac{m-1}{2}})$$ and $$y\coloneqq\diag(1_{\frac{m-1}{2}},-1_{\frac{m-1}{2}},1,-1_{\frac{m-1}{2}},1_{\frac{m-1}{2}}).$$ However, $x$ and $y$ are conjugate and so $xy$ is of spinor norm one. If $q\equiv 1$ modulo $4$, let $\alpha$ be a square root of $-1$ and observe that $g_0=x^2$ where $x=\diag(\alpha 1_{m-1},1,-\alpha 1_{m-1})\in\SO_{2m-1}^\circ(q)$. Hence $g_0$ again has spinor norm one.
		
		Set now $k(\lambda)$ to be the Eichler transformation
		$$
		k(\lambda)\coloneqq\begin{pmatrix}
			1 & \cdots & \lambda & 0\\
			0 & \ddots & 0 & -\lambda\\
			\vdots & \ddots & \ddots & \vdots\\
			0 & \cdots & 0 & 1
		\end{pmatrix} = 1_{2m-1} + \lambda h
		$$
		with $h=e_{1,2m-2} - e_{2,2m-1}.$ This element from $\SO^\circ_{2m-1}(q)$ again is a square of an element from $\SO^\circ_{2m-1}(q)$ for $q$ odd, namely of $k(\lambda/2)$, so has spinor norm $1$.
		For $x=(x_{i,j})_{i,j=1}^{2m-1}$ we compute 
		$$
		x^{-1} = \Omega x^\top \Omega = (x_{2m-j,2m-i})_{i,j=1}^{2m-1}
		$$
		as $\Omega=\Omega^{-1}$.
		We obtain 
		$$
		g_0^x = x^{-1} g_0 x = (-\delta_{i,j} + 2 x_{m,2m-i} x_{m,j})_{i,j=1}^{2m-1},
		$$
		since
		$$
		x^{-1}e_{m,m}x =\sum_{i,k} x_{2m-k,2m-i} e_{i,k} e_{m,m} \cdot \sum_{l,j} x_{l,j} e_{l,j} = \sum_{i,j} x_{m,2m-i} x_{m,j} e_{i,j}.
		$$
		
		Then, according to \cite{tomanov1985generalized}*{pages~41 and~42}, we have the matrix identity
		$$
		r(\lambda,x)r(\mu,x)=r(\mu,x)r(\lambda,x),
		$$
		where $r(\lambda,x)=g_0^xk(\lambda)g_0^x k(-\lambda)$. Let's compute: we see that
		$$
		r(\lambda,x)= g_0^x (1+\lambda h)g_0^x (1-\lambda h) = 1 + \lambda g_0^x h g_0^x - \lambda h -\lambda^2 g_0^x h g_0^x h.
		$$
		
		Now, using $h^2=0_{2m-1}$ repeatedly, we get:
		\begin{align*}
			& r(\lambda,x)r(\mu,x)\\
			& = (1_{2m-1} + \lambda g_0^x h g_0^x -\lambda h - \lambda^2 g_0^x h g_0^x h) \cdot (1_{2m-1} + \mu g_0^x h g_0^x -\mu h - \mu^2 g_0^x h g_0^x h) \\
			& = 1_{2m-1} + \mu g_0^x h g_0^x -\mu h- \mu^2 g_0^x h g_0^x h \\
			&\quad + \lambda g_0^x h g_0^x(1_{2m-1} + \mu g_0^x h g_0^x -\mu h- \mu^2 g_0^x h g_0^x h) \\
			&\quad -\lambda h(1_{2m-1} + \mu g_0^x h g_0^x -\mu h - \mu^2 g_0^x h g_0^x h)\\
			&\quad - \lambda^2 g_0^x h g_0^x h(1_{2m-1} + \mu g_0^x h g_0^x - \mu h- \mu^2 g_0^x h g_0^x h)\\
			& = 1_{2m-1} + \mu g_0^x h g_0^x -\mu h - \mu^2 g_0^x h g_0^x h \\
			&\quad + \lambda g_0^x h g_0^x +  \lambda\mu g_0^x h g_0^x g_0^x h g_0^x -\lambda \mu g_0^x h g_0^xh - \lambda\mu^2 \lambda g_0^x h g_0^x g_0^x h g_0^x h\\
			&\quad -\lambda h - \lambda\mu h g_0^x h g_0^x + \lambda\mu h^2+ \lambda\mu^2 h g_0^x h g_0^x h\\
			&\quad - \lambda^2 g_0^x h g_0^x h - \lambda^2\mu g_0^x h g_0^x h g_0^x h g_0^x + \lambda^2\mu g_0^x h g_0^x hh +  \lambda^2\mu^2 g_0^x h g_0^x h g_0^x h g_0^x h\\
			& = 1_{2m-1} +(\lambda+\mu)(g_0^x h g_0^x-h)   - (\lambda^2+\mu^2) g_0^x h g_0^x h -\lambda\mu g_0^x h g_0^x h\\  
			&\quad- \lambda\mu h g_0^x h g_0^x + \lambda \mu^2 h g_0^x h g_0^x h - \lambda\mu g_0^x h g_0^x h g_0^x h g_0^x + \lambda^2\mu^2 g_0^x h g_0^x h g_0^x h g_0^x h
		\end{align*}
		Thus, we get $r(\lambda,x)r(\mu,x)=r(\mu,x)r(\lambda,x)$ if and only if
		$h g_0^x h g_0^x h =0_{2m-1}.$ Using our formula for $g_0^x$, we get:
		
		\begin{align*}
			& h g_0^x h \\
			&= (e_{1,2m-2} - e_{2,2m-1})g_0^x (e_{1,2m-2} - e_{2,2m-1}) \\
			&= 2 (e_{1,2m-2} - e_{2,2m-1})\\
			&\quad\cdot (x_{m,2}x_{m,1} e_{2m-2,1} + x^2_{m,2} e_{2m-2,2} + x^2_{m,1} e_{2m-1,1} + x_{m,1}x_{m,2}e_{2m-1,2})\\
			&\quad\cdot (e_{1,2m-2} - e_{2,2m-1})\\
			&= 2(x_{m,2}x_{m,1} e_{1,2m-2} -x^2_{m,2}e_{1,2m-1} + x_{m,1}x_{m,2} e_{2,2m-1} -x^2_{m,1}e_{2,2m-2})
		\end{align*}
		Here we use $m\geq3$.
		And hence:
		\begin{align*}
			& h g_0^x h g_0^x h\\
			&= 2(x_{m,2}x_{m,1} e_{1,2m-2} -x^2_{m,2}e_{1,2m-1} + x_{m,1}x_{m,2} e_{2,2m-1} -x^2_{m,1}e_{2,2m-2})\\
			&\quad\cdot g_0^x (e_{1,2m-2} - e_{2,2m-1}) \\
			&= 4(x_{m,2}x_{m,1} e_{1,2m-2} -x^2_{m,2}e_{1,2m-1} + x_{m,1}x_{m,2} e_{2,2m-1} -x^2_{m,1}e_{2,2m-2})\\
			&\quad\cdot (x_{m,2}x_{m,1} e_{2m-2,1} + x^2_{m,2} e_{2m-2,2} + x^2_{m,1} e_{2m-1,1} + x_{m,1}x_{m,2}e_{2m-1,2})\\
			&\quad\cdot (e_{1,2m-2} - e_{2,2m-1})\\
			&= 4(x_{m,2}x_{m,1} e_{1,2m-2} -x^2_{m,2}e_{1,2m-1} + x_{m,1}x_{m,2} e_{2,2m-1} -x^2_{m,1}e_{2,2m-2})\\
			&\quad\cdot (x_{m,2}x_{m,1} e_{2m-2,2m-2} - x^2_{m,2} e_{2m-2,2m-1}\\ 
			&\quad+ x^2_{m,1} e_{2m-1,2m-2}-x_{m,1}x_{m,2}e_{2m-1,2m-1})\\
			&= 0_{2m-1}
		\end{align*}
		
		This shows that there is also a mixed identity $w(x)=[r(\lambda,x),r(\mu,x)]$ of constant length in the orthogonal groups $\POmega^\circ_{2m-1}(q)$ (for $m\geq 3$) of odd degree for $m$ odd or $q\equiv 1$ mod $4$.
	\end{proof}
	
	The above proof does not work for $m=2$, i.e.\ for $\POmega^\circ_3(q)\cong\PSL_2(q)$. In this case, we have $2m-2=2$, so that the computations of the matrix products above are different. Basically, the two $2\times 2$-blocks overlap.
	
	\begin{remark} \label{rem:pso}
		The element $g_0$ defined above lies in $\PSO_{2m-1}^{\circ} (q)$, irrespective of the value of $m$ or $q$. The preceding argument therefore yields a mixed identity of bounded length for $\PSO_{2m-1} ^{\circ} (q)$, for all $m \geq 3$ and $q$ odd. It is as yet unclear whether $\POmega_{2m-1} ^{\circ} (q)$ has a mixed identity of bounded length in the case of $m$ even and $q \equiv 3$ mod $4$.
	\end{remark}
	
	\section{The projective special unitary groups \texorpdfstring{$\PSU_n(q)$}{PSUnq}}
	
	\subsection{Proof of the upper bound in Theorem~\ref{thm:unitry_groups_main_thm}} Here we proceed as in the proof of Lemma~\ref{lem:psl_up_bd}:
	
	\begin{lemma}\label{lem:psu_n_up_bd}
		There is a mixed identity of length $O(q^2)$ for $\PSU_n(q)$.
	\end{lemma}
	
	\begin{proof}
		Choose a unitary transvection $k\in\SU_n(q)$ (see \cite{wilson2009finite}, page~67) and proceed as in the proof of Lemma~\ref{lem:psl_up_bd}. Again, $k$ fixes a hyperplane $H$ and $k^g$ for $g\in\SU_n(q)$ fixes the hyperplane $H.g$ pointwise, so that both fix the codimension-two subspace $U=H\cap H.g\leq V\cong\finfield_{q^2}^n$ pointwise. The rest is the same argument as in the proof of Lemma~\ref{lem:psl_up_bd}, noting that we are in $\SL_n(q^2)$.
	\end{proof}
	
	\subsection{Proof of the lower bound in Theorem~\ref{thm:unitry_groups_main_thm}}
	
	Again, we start by just considering $\PSU_2(q)$ to get an idea of how the proof for $\PSU_n(q)$ ($n\geq 3$) might work. In the proof of the following lemma, we use the ideas from the proof of Lemma~\ref{lem:psl_2_fnl_lem}. Actually, since $\PSL_2(q)\cong\PSU_2(q)$, the two lemmas nearly have the same content, apart from the different groups of constants.
	
	\begin{lemma}\label{lem:lw_bd_psu_2}
		Assume $w\in\GL_2(q^2)\ast\gensubgrp{x}$ is of length $0<l\leq q/2+1$ such that $\overline{w}\in\PGL_2(q^2)\ast\gensubgrp{x}$ is of positive length. Then $\overline{w}$ is non-constant on $\PSU_2(q)$.
	\end{lemma}
	
	\begin{proof}
		Let $f$ be the standard non-singular hermitian form on $V\cong\finfield_{q^2}^2$ with respect to the Frobenius $\finfield_{q^2}\to\finfield_{q^2}$; $\alpha\mapsto\alpha^q$. Then $x\mapsto k(\lambda,x)=x+\lambda f(x,v)v=(1_V+\lambda h)(x)$ defines an element of the general unitary group $\GU_2(q)$ when $\tr(\lambda)=0$ and $f(v,v)=0$ ($\lambda\in\finfield_{q^2}$, $v\in V$). Indeed, it is a \emph{unitary transvection}:
		\begin{align*}
			f(k(\lambda,x),k(\lambda,y))
			&=f(x+\lambda f(x,v)v,y+\lambda f(y,v)v)\\
			&=f(x,y)+\lambda f(x,v)f(v,y)+\lambda^q f(y,v)^q f(x,v)\\
			&\quad+\lambda^{q+1}f(x,v)f(y,v)^q f(v,v)\\
			&=f(x,y)+\tr(\lambda)f(y,v)^q f(x,v)+0=f(x,y).
		\end{align*}
		Here $f$ is semi-linear in the second entry. Indeed, $x\mapsto k(\lambda,x)$ is an element of $\SU_2(q)$ as it has determinant one.
		
		Proceed as in the proof of the lower bound for $\PSL_2(q)$. Choose $\alpha\in\ker(\tr)\setminus\set{0}$ and set $\lambda\coloneqq\alpha\mu$ for $\mu\in\finfield_q$ arbitrary. Note that this parametrizes the kernel of the trace map $\tr\colon\finfield_{q^2}\to\finfield_q$. Consider the word 
		$$
		w=c_0 x^{\varepsilon(1)}c_1\cdots c_{l-1} x^{\varepsilon(l)} c_l\in\GL_2(q^2)\ast\gensubgrp{x}
		$$ 
		and replace it by 
		$$w'=x^{\varepsilon(1)}c_1\cdots c_{l-1} x^{\varepsilon(l)}
		$$ 
		which becomes constant at the same time. Again, by Lemma~\ref{lem:red_one_var_cs}, we may assume that all $c_j$ are non-central ($j=1,\ldots,l-1$). We are looking for a non-trivial isotropic vector $v\in V\cong\finfield_{q^2}^2$ such that $hc_j h\neq0$ for all $j=1,\ldots,l-1$. This means, according to the above definition of $h$, $f(f(x,v)v.c_j,v)v=f(x,v)f(v.c_j,v)v\neq 0$, i.e.\ $f(v.c_j,v)\neq 0$. But since $v$ is isotropic, this holds precisely, when $v$ is not an eigenvector of $c_j$ ($j=1,\ldots,l-1$). However, the $c_j$ altogether have at most $2(l-1)$ eigenspaces of dimension one, since each of them is non-central. Moreover, there are precisely $q+1$ one-dimensional isotropic subspaces. Indeed, $x^{q+1}+y^{q+1}=0$ has exactly $q+1$ solutions, as it is equivalent to $(x/y)^{q+1}=-1$, since $x,y\neq 0$ and the norm $\N\colon\finfield_{q^2}^\times\to\finfield_q^\times; \alpha\mapsto\alpha^{q+1}$ is $q+1:1$ and surjective. But by assumption $2(l-1)<q+1$, so that there is a legal choice for $v$.
		
		Then we plug in $x\mapsto k(\alpha\mu,x)$ into $w'$ and get a polynomial of degree $l$ with $q>l>0$ in $\mu\in\finfield_q$ for $q>2$. Applying Lemma~\ref{lem:poly_non_const} for $q>2$ and Lemma~\ref{lem:no_shrt_ids} for $q=2$, we conclude that $\overline{w}'$ and hence $\overline{w}$ cannot be constant. Note here that the same proof applies to $\PSp_2(q)\cong\PSL_2(q)$ with a slight variation. But also $\PSU_2(q)\cong\PSL_2(q)$, so this is just another proof of Lemma~\ref{lem:psl_2_fnl_lem}.
	\end{proof}
	
	For the proof of the lower bound for $\PSU_n(q)$, we need the following auxiliary lemma on the number of isotropic vectors that a space $V\cong\finfield_{q^2}^n$ with non-zero hermitian form on it can admit. Its proof is standard and can be found in \cite{wilson2009finite}, page~65.
	
	\begin{lemma}\label{lem:num_isotrpc_vects}
		The number of non-zero isotropic vectors of a space $V\cong\finfield_{q^2}^n=\finfield_{q^2}^{k+l}$, with the non-zero hermitian form $f$ on it, is equal to
		$$
		N_{k,l,q}=(q^k-(-1)^k)(q^{k-1}-(-1)^{k-1})q^{2l}+q^{2l}-1,
		$$
		where $\dim(\rad(f))=l<n$ and $n=\dim(V)=k+l$. Set $N_{n,q}\coloneqq N_{n,0,q}$. The expression $N_{k,l,q}$ is equal to $q^{2n-1}+O(q^{2(n-1)})$ for $k\geq 2$. For $k=1$, it is $q^{2(n-1)}+O(1)$.
	\end{lemma}
	
	The key to the proof of the lower bound for $\PSU_n(q)$ is the following general observation concerning the vanishing sets of sesquilinear forms. 
	
	\begin{lemma}
		Let $V\cong\finfield_{q^2}^n$ and $f\colon V\times V\to\finfield_{q^2}$ be the standard unitary form $f(u,v)=\sum_{i=1}^n{u_i v_i^q}$ on $V$. Moreover, let $g\colon V\times V\to\finfield_{q^2}$ be a non-degenerate sesquilinear form such that $g(u,v)=\sum_{i,j=1}^n{c_{ij}u_i v_j^q}$ so that $(c_{ij})_{i,j=1}^n\neq \lambda 1_V$ (for all $\lambda\in\finfield_{q^2}^\times$) is non-scalar. Set $V(f)\coloneqq\set{v\in V\setminus\set{0}}[f(v,v)=0]$. Then:
		$$
		\frac{\card{V(f)\cap V(g)}}{\card{V(f)}}\leq\frac{2}{q}+O(1/q^2).
		$$
	\end{lemma}
	
	In other words, $V(f)$ and $V(g)$ have few points in common.
	
	\begin{proof}
		Assume w.l.o.g.\ that $c_{21}\neq 0$. Indeed, if there is no $c_{ij}\neq 0$ for $i\neq j$ (in which case we could permute the coordinates so that $(i,j)=(2,1)$ and hence $c_{21}\neq 0$), then $(c_{ij})_{i,j=1}^n$ is a diagonal matrix with not all diagonal entries equal to each other. Again, by permuting the coordinates, we may assume that $c_{11}=\lambda\neq\mu=c_{22}$. Choose two non-zero elements $a,b\in\finfield_{q^2}$ such that $a^{q+1}+b^{q+1}=1$. This is possible, since the norm $\N\colon\finfield_{q^2}^\times\to\finfield_q^\times;\alpha\mapsto\alpha^{q+1}$ is surjective. Then $u=\left(\begin{smallmatrix} a & b\\ -b^q & a^q\end{smallmatrix}\right)$ is an element of $\SU_2(q)$: 
		$$
		uu^\ast=\begin{pmatrix} a & b\\ -b^q & a^q\end{pmatrix}\begin{pmatrix} a^q & -b\\ b^q & a\end{pmatrix}=\begin{pmatrix} a^{q+1}+b^{q+1} & 0\\ 0 & a^{q+1}+b^{q+1}\end{pmatrix}=1_2.
		$$
		Now we compute
		\begin{align*}
			u
			\begin{pmatrix}
				\lambda & 0\\
				0 & \mu
			\end{pmatrix}
			u^\ast
			&=
			\begin{pmatrix} 
				a & b\\ 
				-b^q & a^q
			\end{pmatrix}
			\begin{pmatrix}
				\lambda & 0\\
				0 & \mu
			\end{pmatrix}
			\begin{pmatrix} 
				a^q & -b\\ 
				b^q & a
			\end{pmatrix}\\
			&=
			\begin{pmatrix} 
				\lambda a & \mu b\\ 
				-\lambda b^q & \mu a^q
			\end{pmatrix}
			\begin{pmatrix} 
				a^q & -b\\ 
				b^q & a
			\end{pmatrix}\\
			&=
			\begin{pmatrix}
				\lambda a^{q+1}+\mu b^{q+1} & ab(\mu-\lambda)\\
				a^q b^q(\mu-\lambda) & \mu a^{q+1}+\lambda b^{q+1}
			\end{pmatrix}.
		\end{align*}
		Since $\lambda\neq\mu$, the two off-diagonal matrix entries are non-zero and we can conjugate $(c_{ij})_{i,j=1}^n$ by $u\oplus 1_{n-2}$ to get $c_{21}\neq 0$, while we preserve the form $f$.
		
		Let $v\in V$ be isotropic with respect to $f$ and $v_1\neq 0$. There are exactly $N_{n,q}-N_{n-1,q}$ such vectors. Assume $v$ is isotropic with respect to $g$ as well. Then $v.\lambda=(\lambda v_1,v_2,\ldots,v_n)$ for $\lambda\in\finfield_{q^2}$, $\lambda^{q+1}=1$, is isotropic for $f$, too. This defines an action of the cyclic group $C=\ker(\N\colon\finfield_{q^2}^\times\to\finfield_q^\times)=\set{\alpha\in\finfield_{q^2}}[\alpha^{q+1}=1]$ on the points of $V(f)$. In order that $v.\lambda$ is isotropic for $g$ as well, we must have:
		\begin{align*}    
			0&=g(v.\lambda,v.\lambda)-g(v,v) \\
			&=\lambda^{q+1}c_{11}v_1^{q+1}-c_{11}v_1^{q+1}+(\lambda-1)\sum_{i=2}^n{ c_{1i} v_1v_i^q}+(\lambda^q-1)\sum_{i=2}^n{c_{i1} v_iv_1^q} \\
			&=0+\lambda^q\sum_{i=2}^n{c_{i1} v_iv_1^q}+\lambda\sum_{i=2}^n{ c_{1i} v_1v_i^q}-\sum_{i=2}^n{(c_{i1} v_iv_1^q+c_{1i} v_1v_i^q)} \\
			&=a\lambda^q+b\lambda-a-b \\
			&=a\lambda^{-1}+b\lambda-a-b.
		\end{align*}
		This equation has at most two solutions in $\lambda$ when $a$ and $b$ are not both zero (indeed, these are $1$, and $a/b$ when $a,b\neq 0$). In the opposite case, $a=0$, so $v$ lies in the kernel $U=\ker(\varphi)$ of the non-zero (since $c_{21}\neq0$) linear functional $\varphi\colon v\mapsto\sum_{i=2}^n{c_{i1} v_i}$. The space $U$ cannot be totally isotropic with respect to $f$, since $\dim(U)=n-1$ and $n\geq3$. Set $k\coloneqq n-1-\dim(\rad(\rest{f}_U))\geq 1$. According to Lemma~\ref{lem:num_isotrpc_vects}, there are 
		$$
		N_{k,n-1-k,q}=
		\begin{cases}
			q^{2(n-2)}+O(1) & \text{for } k=1\\
			q^{2(n-1)-1}+O(q^{2(n-2)}) & \text{for } k\geq2
		\end{cases}
		$$ 
		such non-zero vectors $v$. Hence we can estimate the cardinality of $V(f)\cap V(g)$ as follows:
		$$
		\card{V(f)\cap V(g)}\leq N_{n-1,q}+N_{k,n-1-k,q}+\frac{2}{q+1}(N_{n,q}-N_{n-1,q}-N_{k,n-1-k,q}).
		$$
		If $k=1$, applying Lemma~\ref{lem:num_isotrpc_vects}, we obtain
		\begin{align*}
			\card{V(f)\cap V(g)}
			&\leq q^{2(n-1)-1}+O(q^{2(n-2)})+q^{2(n-2)}+O(1)\\
			&\quad+\frac{2}{q+1}(q^{2n-1}+O(q^{2(n-1)})-q^{2(n-1)-1}\\
			&\quad+O(q^{2(n-2)})-q^{2(n-2)}+O(1))\\
			&=2q^{2(n-1)}+O(q^{2(n-1)-1}).
		\end{align*}
		Similarly, for $k\geq 2$, we get
		\begin{align*}
			\card{V(f)\cap V(g)}
			&\leq 2q^{2(n-1)-1}+O(q^{2(n-2)})\\
			&\quad +\frac{2}{q+1}(q^{2n-1}+O(q^{2(n-1)})-2q^{2(n-1)-1}+O(q^{2(n-2)}))\\
			&=2q^{2(n-1)}+O(q^{2(n-1)-1})
		\end{align*}
		as well.
		Thus,
		\begin{align*}
			\frac{\card{V(f)\cap V(g)}}{\card{V(f)}}
			&\leq\frac{2q^{2(n-1)}+O(q^{2(n-1)-1})}{q^{2n-1}+O(q^{2(n-1)})}=\frac{2}{q}+O(1/q^2).
		\end{align*}
		The proof is complete.
	\end{proof}
	
	\begin{lemma}\label{lem:lw_bd_psu_n}
		Assume $w\in\GL_n(q^2)\ast\gensubgrp{x}$ ($n\geq 3$) is of length $0<l\leq q/2+O(1)$ such that $\overline{w}\in\PGL_n(q^2)\ast\gensubgrp{x}$ is of positive length. Then $\overline{w}$ is non-constant on $\PSU_n(q)$.
	\end{lemma}
	
	\begin{proof}
		We proceed as in the proof of Lemma~\ref{lem:lw_bd_psu_2}. We have to make sure that there is a vector $v\in V$ such that $f(v,v)=0$ and $g_j(v,v)\coloneqq f(v.c_j,v)\neq 0$ for $j=1,\ldots,l-1$. But by the previous lemma for $g=g_j$ we have
		$$
		\frac{\card{V(f)\cap V(g)}}{\card{V(f)}}\leq\frac{2}{q}+O(1/q^2),  
		$$
		and $1/(2/q+O(1/q^2))=q/2+O(1)$, so that $V(f)\setminus\bigcup_{j=1}^{l-1}{V(g_j)}\neq\emptyset$. The proof is complete.
	\end{proof}
	
	\section{Outlook and further comments} \label{comments}
	
	We note that the mixed identities $w\in G\ast\freegrp_r$ considered herein for $G=S_n$ and $A_n$, and most other groups covered in this article, are \emph{singular}, i.e.\ they lie in the kernel of the augmentation map $\varepsilon\colon G\ast\freegrp_r\to\freegrp_r$ which fixes $\freegrp_r$ element-wise and maps $G\ni g\mapsto 1_{\freegrp_r}$, following the terminology introduced in \cite{klyachkothom2017new}, their \emph{content} is trivial. By Theorem~1 in \cite{schneiderthom2022word} the former is necessary for $S_n$, as, by this theorem, there are no non-singular identities of bounded length. We will address this question in forthcoming work for quasi-simple groups of Lie type, \cite{bradford2023nonsingular}.
	
	Let us come back to the case $\POmega^\circ_{2m-1}(q)$, which we cover for even $m$ only when $q \equiv 1$ mod $4$. The case $q \equiv 3$ mod $4$ is rather peculiar. It seems plausible and likely that there is no mixed identity of bounded length in this case, even though the almost simple group $\PSO_{2m-1}^\circ(q)$ including also the elements of non-trivial spinor norm does satisfy a mixed identity of bounded length, see Remark \ref{rem:pso}. This shows even more drastically then for $\PSL_n(q)$ that passage to an almost simple group might change the asymptotics of the length of shortest mixed identities.
	
	In a forthcoming work, we plan to address the remaining families of simple groups of Lie type of bounded rank.

\end{document}